\documentclass[a4paper]{article}

\usepackage[all]{xy}\usepackage[latin1]{inputenc}        
\usepackage[dvips]{graphics,graphicx}
\usepackage{amsfonts,amssymb,amsmath,color,mathrsfs, amstext}
\usepackage{amsbsy, amsopn, amscd, amsxtra, amsthm,authblk}
\usepackage{upref}
\usepackage[colorlinks,
            linkcolor=red,
            anchorcolor=red,
            citecolor=red
            ]{hyperref}

\usepackage{geometry}
\usepackage{float}

\usepackage{yhmath}

\def\e{\epsilon}

\numberwithin{equation}{section}              
\newtheorem{theorem}{Theorem}[section]
\newtheorem{lemma}{Lemma}[section]
\newtheorem{proposition}{Proposition}[section]
\newtheorem*{proposition*}{Proposition}

\newtheorem*{corollary*}{Corollary}
\newtheorem{definition}{Definition}[section]
\newtheorem*{definitions*}{Definitions}
\newtheorem*{conjecture*}{\bf Conjecture}

\newtheorem*{example*}{\bf Example}
\theoremstyle{remark}
\newtheorem{remark}{\bf Remark}[section]

\begin{document}
\date{}                                     
\title{A dispersive regularization for the modified Camassa-Holm equation}

\author[1,2,3]{Yu Gao\thanks{yugao@hit.edu.cn}}
\author[2]{Lei Li\thanks{leili@math.duke.edu}}
\author[2,3]{Jian-Guo Liu\thanks{jliu@phy.duke.edu}}
\affil[1]{Department of Mathematics, Harbin Institute of Technology, Harbin, 150001, P.R. China.}
\affil[2]{Department of Mathematics, Duke University, Durham, NC 27708, USA.}
\affil[3]{ Department of Physics, Duke University, Durham, NC 27708, USA.}

\maketitle

\begin{abstract}
In this paper, we present a dispersive regularization for the modified Camassa-Holm equation (mCH) in one dimension, which is achieved through a double mollification for the system of ODEs describing trajectories of $N$-peakon solutions. From this regularized system of ODEs, we obtain approximated $N$-peakon solutions with no collision between peakons. Then, a global $N$-peakon solution for the mCH equation is obtained, whose trajectories are global Lipschitz functions and do not cross each other. When $N=2$, the limiting solution is a sticky peakon weak solution. By a limiting process, we also derive a system of ODEs to describe $N$-peakon solutions.  At last, using the $N$-peakon solutions and through a mean field limit process, we obtain global weak solutions for general initial data $m_0$ in Radon measure space.
\end{abstract}

\section{Introduction}
This work is devoted to investigate the $N$-peakon solutions to the following modified Camassa-Holm (mCH) equation with cubic nonlinearity:
\begin{align}\label{mCH}
m_t+[(u^2-u^2_x)m]_x=0,\quad m=u-u_{xx},\quad x\in\mathbb{R},~~t>0,
\end{align}
subject to the initial condition
\begin{align}\label{initial m}
m(x,0)=m_0(x),\quad x\in\mathbb{R}.
\end{align}
From the fundamental solution $G(x)=\frac{1}{2}e^{-|x|}$ to the Helmholtz operator $1-\partial_{xx}$, function $u$ can be written as a convolution of $m$ with the kernel $G$:
$$u(x,t)=\int_{\mathbb{R}}G(x-y)m(y)dy.$$
In the mCH equation, the shape of function $G$ is referred to as a peakon at $x=0$ and the mCH equation has weak solutions (see Definition \ref{weaksolution}) with $N$ peakons, which are of the form \cite{GaoLiu,GuiLiuOlverQu}:
\begin{align}\label{eq:Npeakon}
u^N(x,t)=\sum_{i=1}^Np_iG(x-x_i(t)),~~m^N(x,t)=\sum_{i=1}^Np_i\delta(x-x_i(t)),
\end{align}
where $p_i$ ($1\leq i\leq N$) are constant amplitudes of peakons. We call this kind of weak solutions as $N$-peakon solutions.
When $x_1(t)<x_2(t)<\cdots<x_N(t)$, trajectories $x_i(t)$ of $N$-peakon solutions in \eqref{eq:Npeakon} satisfies  \cite{GaoLiu,GuiLiuOlverQu}:
\begin{equation}\label{eq:Npeakonmonoticity0}
\frac{d}{dt}x_i=\frac{1}{6}p_i^2+\frac{1}{2}\sum_{j< i}p_ip_je^{x_j-x_i}+\frac{1}{2}\sum_{j> i}p_ip_je^{x_i-x_j}+\sum_{1\leq m<i<n\leq N}p_mp_ne^{x_m-x_n}.
\end{equation} 
In general, solutions $\{x_i(t)\}_{i=1}^N$ to \eqref{eq:Npeakonmonoticity0} will collide with each other in finite time (see Remark \ref{rmk:nonuniqueness}).
By the standard ODE theories, we know that \eqref{eq:Npeakonmonoticity0} has global solutions $\{x_i(t)\}_{i=1}^N$ subject to any initial data $\{x_i(0)\}_{i=1}^N$.  However, $u^N(x,t)$ constructed by \eqref{eq:Npeakon} with  global solutions $\{x_i(t)\}_{i=1}^N$ to \eqref{eq:Npeakonmonoticity0} is not a weak solution to the mCH equation after the first collision time (see Remark \ref{rmk:rightequation}). There are some nature questions:
\begin{description}
	\item [(i)] What will be a weak solution to the mCH equation after collisions? Is it unique? If not unique, what is the selection principle?
	\item [(ii)] If there is a weak solution to the mCH equation after collisions, is it still in the form of $N$-peakon solutions (peakons can be coincide)?
	\item [(iii)] If the weak solution is still a $N$-peakon solution after collision, how do peakons evolve? In other words, do they stick together, cross each other, or scatter?
\end{description}
Paper \cite{GaoLiu} showed global existence and nonuniqueness of weak solutions when initial data $m_0\in\mathcal{M}(\mathbb{R})$ (Radon measure space), which partially answered question (i). After collision, all the situations mentioned in the above question (iii) can happen (see Remark \ref{rmk:nonuniqueness}). 

In this paper, we will study these questions through a dispersive regularization for the following reasons.  
\begin{description}
	\item[(i)] This dispersive regularization could be a candidate for the selection principle.
	\item[(ii)] As described below, if initial datum is of $N$-peakon form, then the regularized solution $u^{N,\e}$ is also of $N$-peakon form, and so is the limiting $N$-peakon solution.
\end{description}

The main purpose of this paper is to study the behavior of $\e\to0$ limit for the dispersive regularization. First, we introduce the dispersive regularization for the mCH equation.

To illustrate the dispersive regularization method clearly,  we start with one peakon solution $pG(x-x(t))$ (solitary wave solution). We know that $pG(x-x(t))$ is a weak solution if and only if the traveling speed is $\frac{d}{dt}x(t)=\frac{1}{6}p^2$ \cite[Proposition 4.3]{GaoLiu}. Because  characteristics  equation for \eqref{mCH} is given by
\begin{align}\label{characterics}
\frac{d}{dt}x(t)=u^2(x(t),t)+u_x^2(x(t),t),
\end{align}
for solution $pG(x-x(t))$ we obtain
\begin{align}\label{travelling speed}
\frac{d}{dt}x(t)=p^2G^2(0)-p^2(G_x^2)(0)=\frac{1}{6}p^2.
\end{align}
\eqref{travelling speed} implies that to obtain solitary wave solutions, the correct definition of $G_x^2$ at $0$  is given by
\begin{align}\label{correctdefinition}
(G^2_x)(0)=G^2(0)-\frac{1}{6}=\frac{1}{12}.
\end{align}
However,  $G_x^2$ is a BV function which has a removable discontinuity at  $0$ and 
\begin{align}\label{removable}
(G_x^2)(0-)=(G_x^2)(0+)=\frac{1}{4},
\end{align}
which is different with \eqref{correctdefinition}. 
To understand the discrepancy between \eqref{correctdefinition} and \eqref{removable}, our strategy is to use the dispersive regularization and the limit of the regularization.
Mollify $G(x)$ as
$$G^\epsilon(x):=(\rho_\epsilon\ast G)(x),$$
where $\rho_\epsilon$ is a mollifier that is even (see Definition \ref{defmollifier}). Then, we can obtain \eqref{correctdefinition} in the limiting process (Lemma \ref{convergencetoNpeakon}):
\begin{align}\label{limitingprocess}
\lim_{\epsilon\rightarrow0}(\rho_\epsilon\ast(G_x^\epsilon)^2)(0)=\frac{1}{12}.
\end{align}
The above limiting process is independent of the mollifier $\rho_\e$.

Naturally, we generalize this dispersive regularization method to $N$-peakon solutions  $u^N(x,t)=\sum_{i=1}^N p_iG(x-x_i(t))$. From the characteristic equation \eqref{characterics}, we formally obtain the system of ODEs for $x_i(t)$ 
\begin{align}\label{formequation}
\frac{d}{dt}x_i(t)=\big[u^N(x_i(t),t)\big]^2-\big[u_x^N(x_i(t),t)\big]^2,\quad i=1,\ldots,N.
\end{align}
$\big[u_x^N(x,t)\big]^2=\big(\sum_{j=1}^Np_jG_x(x-x_j(t))\big)^2$ is a BV function and it has a discontinuity at $x_i(t)$.  By using similar regularization method in  \eqref{limitingprocess}, we regularize the vector field in \eqref{formequation}. For $\{x_k\}_{k=1}^N$, denote
\begin{align}\label{uusub}
u^{N,\epsilon}(x;\{x_k\}):=\sum_{i=1}^Np_iG^\epsilon(x-x_i)~\textrm{ and }~U^N_\epsilon(x;\{x_k\}):=\big[u^{N,\epsilon}\big]^2-\big[u_x^{N,\epsilon}\big]^2.
\end{align}
The dispersive regularization for $N$ peakons is given by
\begin{align}\label{approximateODE21}
\frac{d}{dt}x^\epsilon_i(t)=U^{N,\epsilon}(x^\epsilon_i(t);\{x^\epsilon_k(t)\}):=(\rho_\epsilon\ast U^N_\epsilon)(x_i^\epsilon(t);\{x_k^\epsilon(t)\}),\qquad i=1,\ldots,N.
\end{align}
The above regularization method is subtle. We emphasize that if we use $U^N_\epsilon$ given by \eqref{uusub} as a vector field (which is already global Lipschitz) instead of $U^{N,\epsilon}$, then comparing with \eqref{limitingprocess} we have
$$\lim_{\epsilon\rightarrow0}(G^\epsilon_x)^2(0)=0.$$
In this case,  the traveling speed of the soliton (one peakon) is given by 
$$\frac{d}{dt}x(t)=p^2G^2(0)-p^2(G_x^2)(0)=\frac{1}{4}p^2,$$
which is different with the correct speed $\frac{1}{6}p^2$ for one peakon solution.

By solutions to \eqref{approximateODE21}, we construct approximate $N$-peakon solutions to \eqref{mCH} as:
$$u^{N,\epsilon}(x,t):=\sum_{i=1}^Np_iG^\e(x-x_i^\e(t)).$$
Let $\e\to0$ in $u^{N,\epsilon}(x,t)$ and we can obtain a $N$-peakon solution 
\begin{align}\label{eq:sollimit}
u^{N}(x,t)=\sum_{i=1}^Np_iG(x-x_i(t)),
\end{align}
to the mCH equation, where $x_i(t)$ are Lipschitz functions (see Theorem \ref{thm:fixedN}).

If we fix $N$ and let $\epsilon$ go to $0$ in the regularized system of ODEs \eqref{approximateODE21}, we can obtain a limiting ($\epsilon\rightarrow 0$ in the sense described in Proposition \ref{discontinue}) system of ODEs to describe  $N$-peakon solutions.  
\begin{footnotesize}
	\begin{align}\label{eq:NpeakontrajectoriesODE}
	\frac{d}{dt}x_i(t)=\left(\sum_{j=1}^Np_jG(x_i(t)-x_j(t))\right)^2-\left(\sum_{j\in \mathcal{N}_{i1}(t)}p_j G_x(x_i(t)-x_j(t))\right)^2-\frac{1}{12}\left(\sum_{k\in \mathcal{N}_{i2}(t)}p_k\right)^2.
	\end{align}
\end{footnotesize}
The vector field of the above system is not Lipschitz. Solutions for this equation are not unique, which implies peakon solutions to \eqref{mCH} are not unique. The nonuniqueness of peakon solutions was obtained in \cite{GaoLiu}. When $x_1(t)<x_2(t)<\cdots<x_N(t)$, the system of ODEs \eqref{eq:NpeakontrajectoriesODE} is equivalent to \eqref{eq:Npeakonmonoticity0}.

We also prove that trajectories $x_i^\e(t)$ given by \eqref{approximateODE21} never collide with each other (see Theorem \ref{thm:nocollisionepsilon}), which means if $x_1^\e(0)<x_2^\e(0)<\cdots<x_N^\e(0)$, then $x_1^\e(t)<x_2^\e(t)<\cdots<x_N^\e(t)$ for any $t>0$. For the limiting $N$-peakon solutions \eqref{eq:sollimit}, we have $x_1(t)\leq x_2(t)\leq \cdots\leq x_N(t)$. Notice that the sticky $N$-peakon solutions obtained in \cite{GaoLiu} also have this property and in the sticky $N$-peakon solutions, $\{x_i(t)\}_{i=1}^N$ stick together whenever they collide. 
When $N=2$, we prove that peakon solutions given by the dispersive regularization are exactly the sticky peakon solutions (see Theorem \ref{thm:twosticky}).  
However, the situation when $N\ge3$ can be more complicated.  Some of the peakon solutions given by the dispersive regularization are  sticky peakon solutions  (see Figure \ref{fig:threetrajectories1}) and some are not (see Figure \ref{fig:threetrajectories2}). 

For general initial data $m_0\in\mathcal{M}(\mathbb{R})$, we use a mean field limit method to prove global existence of weak solutions to \eqref{mCH} (see Section \ref{sec:meanfield}).

There are also some other interesting properties about the mCH equation, which we list below.

The mCH equation was introduced as a new integrable system by several different researchers \cite{Fokas,Fuchssteiner,Olever}. In a physical context, it was derived from the two-dimensional Euler equation by using a singular perturbation method in which the variable $u$ represents the velocity of fluid \cite{Qiao2}, and Lax-pair was also given in \cite{Qiao2}. The mCH equation has a bi-Hamiltonian structure \cite{GuiLiuOlverQu,Olever} with Hamiltonian functionals
\begin{align}\label{Hamiltonians}
H_0=\int_{\mathbb{R}}mudx,\quad H_1=\frac{1}{4}\int_{\mathbb{R}}\left(u^4+2u^2u_x^2-\frac{1}{3}u^4_x\right)dx.
\end{align}
\eqref{mCH} can be written in the bi-Hamiltonian form \cite{GuiLiuOlverQu,Olever},
$$m_t=-((u^2-u_x^2)m)_x=J\frac{\delta H_0}{\delta m}=K\frac{\delta H_1}{\delta m},$$
where $$J=-\partial_x\Big(m\partial_x^{-1}(m\partial_x)\Big),\quad K=\partial_x^3-\partial_x$$
are compatible Hamiltonian operators. Here $H_0$ and $H_1$ are conserved quantities for smooth solutions. $H_0$ is also a conserved quantity for $W^{2,1}(\mathbb{R})$ weak solutions \cite{GaoLiu}. $N$-peakon solutions are not in the solution class $W^{2,1}(\mathbb{R})$ and $H_0$, $H_1$ are not conserved for $N$-peakon solutions in the case $N\geq2$. This is different with the CH equation:
$$m_t+(um)_x+mu_x=0,\quad m=u-u_{xx},\quad x\in\mathbb{R},~~t>0,$$
which also has $N$-peakon solutions of the form
$$u^N(x,t)=\sum_{i=1}^Np_i(t)e^{-|x-x_i(t)|}.$$
The amplitudes $p_i(t)$ evolves with time which is different with the $N$-peakon solutions to mCH equation \eqref{mCH} where $p_i$ are constants. $p_i(t)$ and $x_i(t)$ satisfy the following Hamiltonian system of ODEs:
\begin{gather}\label{CH Npeakon}
\left\{\begin{split}
&\frac{d}{dt}x_i(t)=\sum_{j=1}^Np_j(t)e^{-|x_i(t)-x_j(t)|},~~i=1,\ldots,N,\\
&\frac{d}{dt}p_i(t)=\sum_{j=1}^Np_i(t)p_j(t)\mathrm{sgn}\big(x_i(t)-x_j(t)\big)e^{-|x_i(t)-x_j(t)|},~~i=1,\ldots,N,
\end{split}
\right.
\end{gather}
and  the Hamiltonian function is given by
$$\mathcal{H}_0(t)=\frac{1}{2}\sum_{i,j=1}^Np_i(t)p_j(t)e^{-|x_i(t)-x_j(t)|},$$
which is a conserved quantity for $N$-peakon solutions and the corresponding functional $H_0$ given by \eqref{Hamiltonians} is conserved for smooth solutions.
When $p_i(0)>0$, there is no collision between $x_i(t)$ \cite{CamassaLee,AlinaLiu}.
In comparison, system \eqref{eq:Npeakonmonoticity0} is a nonautonomous system as described below. Let $\tilde{x}_i(t):=x_i(t)-\frac{1}{6}p_i^2t$. Denote
$$X(t):=(\tilde{x}_1(t),\tilde{x}_2(t),\cdots,\tilde{x}_N(t))^T,$$ 
and
$$\mathcal{H}(X,t):=\sum_{1\leq i<j\leq N}p_ip_je^{x_i(t)-x_j(t)}=\sum_{1\leq i<j\leq N}p_ip_je^{\frac{1}{6}(p_j^2-p_i^2)t+\tilde{x}_i(t)-\tilde{x}_j(t)}.$$ 
Then, \eqref{eq:Npeakonmonoticity0} can be rewritten as a Hamiltonian system:
\begin{align}\label{eq:Hamiltonian}
\frac{dX}{dt}=A\frac{\delta\mathcal{H}}{\delta X},
\end{align}
where 
\begin{align}\label{Hamiltonian discrete}
A=(a_{ij})_{N\times N},~~a_{ij}=\begin{cases}
-\frac{1}{2},~~i<j;\\
0,~~i=j;\\
\frac{1}{2},~~i>j.
\end{cases},\textrm{ and}~\frac{\delta \mathcal{H}}{\delta X}:=\Big(\frac{\partial\mathcal{H}}{\partial \tilde{x}_1},\ldots, \frac{\partial\mathcal{H}}{\partial \tilde{x}_N}\Big).
\end{align}
Notice that $\mathcal{H}$ depends on $t$ and it is not a conservative quantity. 

For more results about local well-posedness and blow up behavior of the strong solutions to  \eqref{mCH} one can refer to  \cite{Chenrongming,FuGui,GuiLiuOlverQu,Himonas,LiuOlver}.
In \cite{Qingtianzhang}, Zhang used the method of dissipative approximation to prove the existence and uniqueness of global entropy weak solutions $u$ in $W^{2,1}(\mathbb{R})$ for the dispersionless mCH equation \eqref{mCH}. 

The rest of this article is organized as follows. In Section \ref{sec:regularization}, we introduce the dispersive regularization in detail and prove global existence of $N$-peakon solutions. By a limiting process, we obtain a system of ODEs to describe $N$-peakon solutions. In Section \ref{sec:stickylim}, we prove that trajectories of $N$-peakon solutions given by dispersive regularization will never cross each other. When $N=2$, the limiting peakon solutions are exactly the sticky peakon solutions. When $N=3$, we present two figures to show two different situations. At last, we use a mean field limit method to prove global existence of weak solutions to \eqref{mCH} for general initial data $m_0\in\mathcal{M}(\mathbb{R})$.

\section{Dispersive regularization and $N$-peakon solutions}\label{sec:regularization}

In this section, we introduce the dispersive regularization in detail and use the regularized ODE system to give approximate solutions. Then, by some compactness arguments we prove global existence of $N$-peakon solutions.
 
 \subsection{Dispersive regularization and weak consistency}

First,  let $\mathcal{S}(\mathbb{R})$ be the Schwartz class of smooth functions to define mollifiers.    $f\in\mathcal{S}(\mathbb{R})$ if and only if $f\in C^\infty(\mathbb{R})$ and for all positive integers $m$ and $n$ 
$$\sup_{x\in\mathbb{R}}|x^mf^{(n)}(x)|<\infty.$$
\begin{definition}\label{defmollifier}
(i). Define the mollifier $0\leq\rho\in \mathcal{S}(\mathbb{R})$ satisfying
\begin{align*}
\int_{\mathbb{R}}\rho(x)dx=1,\quad \rho(x)=\rho(|x|)~\textrm{ for }~x\in\mathbb{R}.
\end{align*}
(ii). For each $\epsilon>0$, set
$$\rho_\epsilon(x):=\frac{1}{\epsilon}\rho(\frac{x}{\epsilon}).$$
\end{definition}

Fix an integer $N>0$. Give an initial data
\begin{align}\label{eq:m0N}
m_0^N(x)=\sum_{i=1}^Np_i\delta(x-c_i),~~c_1<c_2<\cdots<c_N~\textrm{ and }~\sum_{i=1}^N|p_i|\leq M_0,
\end{align}
for some constants $p_i$, $c_i$ ($1\leq i\leq N$) and  $M_0$.

As stated in Introduction, we set $G^\epsilon(x)=(G\ast\rho_\epsilon)(x)$.
For any $N$ particles $\{x_k\}_{k=1}^N\subset\mathbb{R}$, define ($p_k$ is the same as in \eqref{eq:m0N})
$$u^{N,\epsilon}(x;\{x_k\}_{k=1}^N):=\sum_{k=1}^Np_kG^\epsilon(x-x_k),$$
$$U^N_\epsilon(x;\{x_k\}_{k=1}^N):=\left[(u^{N,\epsilon})^2-(\partial_xu^{N,\epsilon})^2\right](x;\{x_k\}_{k=1}^N),$$
and
$$U^{N,\epsilon}(x;\{x_k\}_{k=1}^N):=(\rho_\epsilon\ast U_\epsilon^N)(x;\{x_k\}_{k=1}^N).$$
The system of ODEs for dispersive regularization is given by
\begin{align}\label{approximateODE}
\frac{d}{dt}x_i^\epsilon(t)=U^{N,\epsilon}(x^\epsilon_i(t);\{x_k^\epsilon(t)\}_{k=1}^N),\quad i=1,\cdots,N,
\end{align}
with initial data $x_i^\epsilon(0)=c_i$ given in \eqref{eq:m0N}. This system is equivalent to \eqref{approximateODE21} mentioned in  Introduction.  Because $U^{N,\epsilon}$ is  Lipschitz and bounded,  existence and uniqueness of a global solution $\{x_i^\epsilon(t)\}_{i=1}^N$ to this system of ODEs follow from  standard ODE theories.  By using the solution $\{x_i^\epsilon(t)\}_{i=1}^N$, we set
\begin{align}
u^{N,\epsilon}(x,t):=u^{N,\epsilon}(x;\{x_k^\epsilon(t)\}_{k=1}^N)\label{uNepisilon}
\end{align}
and
\begin{align}
m^{N,\epsilon}(x,t):=\sum_{i=1}^Np_i\rho_\epsilon(x-x_i^\epsilon(t)),\quad m^N_\epsilon(x,t):=\sum_{i=1}^Np_i\delta(x-x_i^\epsilon(t)).\label{mNepisilon}
\end{align}
Due to $(1-\partial_{xx})G^\epsilon=\rho_\epsilon$, we have
\begin{align}\label{m=Gastu}
m^{N,\epsilon}(x,t)=(\rho_{\epsilon}\ast m^N_\epsilon)(x,t)~\textrm{ and }~(1-\partial_{xx})u^{N,\epsilon}(x,t)=m^{N,\epsilon}(x,t).
\end{align}
Set
\begin{align}
U^N_\epsilon(x,t):=U^N_\epsilon(x;\{x_k^\epsilon(t)\}_{k=1}^N),\quad U^{N,\epsilon}(x,t):=U^{N,\epsilon}(x;\{x_k^\epsilon(t)\}_{k=1}^N).\label{UNepisilon}
\end{align}
Therefore,  $U^{N,\epsilon}(x,t)=(\rho_\epsilon\ast U^N_\epsilon)(x,t)$ and \eqref{approximateODE} (or \eqref{approximateODE21}) can be rewritten as
\begin{align}\label{approximateODE2}
\frac{d}{dt}x_i^\epsilon(t)=U^{N,\epsilon}(x^\epsilon_i(t),t),\quad i=1,\cdots,N.
\end{align}

Next, we show that $u^{N,\epsilon}$ defined by \eqref{uNepisilon} is weak consistent with the mCH equation \eqref{mCH}. Let us give the definition of  weak solutions first.
Rewrite \eqref{mCH} as an equation of $u$,
\begin{align*}
&\quad(1-\partial_{xx})u_t+[(u^2-u^2_x)(u-u_{xx})]_x\\
&=(1-\partial_{xx})u_t+(u^3+uu_x^2)_x-\frac{1}{3}(u^3)_{xxx}+\frac{1}{3}(u^3_x)_{xx}=0.
\end{align*}
For test function $\phi\in C_c^\infty(\mathbb{R}\times[0,T))$ ($T>0$), we denote the functional
\begin{align}\label{weakfunctional}
\mathcal{L}(u,\phi):&=\int_0^T\int_{\mathbb{R}}u(x,t)[\phi_t(x,t)-\phi_{txx}(x,t)]dxdt\nonumber\\
&\quad-\frac{1}{3}\int_0^T\int_{\mathbb{R}}u^3_x(x,t)\phi_{xx}(x,t)dxdt-\frac{1}{3}\int_0^T\int_{\mathbb{R}}u^3(x,t)\phi_{xxx}(x,t)dxdt\nonumber\\
&\quad+\int_0^T\int_{\mathbb{R}}(u^3+uu_x^2)\phi_x(x,t)dxdt.
\end{align}
Then, the definition of weak solutions  in terms of $u$ is given as follows.
\begin{definition}\label{weaksolution}
For $m_0\in\mathcal{M}(\mathbb{R})$, a function
\[
u\in C([0,T);H^1(\mathbb{R}))\cap L^\infty(0,T; W^{1,\infty}(\mathbb{R}))
\] 
is said to be a weak solution of the mCH equation if
$$\mathcal{L}(u,\phi)=-\int_{\mathbb{R}}\phi(x,0)dm_0$$
holds for all $\phi\in C_c^\infty(\mathbb{R}\times[0,T))$.  If $T=+\infty$, we call $u$ as a global weak solution of the mCH equation.
\end{definition}

For simplicity in notations, we denote
$$\langle f(x,t), g(x,t)\rangle:=\int_0^\infty\int_{\mathbb{R}}f(x,t)g(x,t)dxdt.$$
With the definitions \eqref{mNepisilon}-\eqref{approximateODE2}, for any $\phi\in C_c^\infty(\mathbb{R}\times[0,T))$, we have
\begin{align}\label{consistence0}
&\quad\langle m^N_\epsilon,\phi_t\rangle+\langle U^{N,\epsilon} m^N_\epsilon,\phi_x\rangle\nonumber\\
&=\int_0^T\int_{\mathbb{R}}\sum_{i=1}^Np_i\delta(x-x_i^\epsilon(t))\phi_t(x,t)dxdt\nonumber\\
&\qquad\qquad\qquad+\int_0^T\int_{\mathbb{R}}\sum_{i=1}^Np_i\delta(x-x_i^\epsilon(t))U^{N,\epsilon}(x,t)\phi_x(x,t)dxdt\nonumber\\
&=\int_0^T\sum_{i=1}^Np_i[\phi_t(x_i^\epsilon(t),t)+U^{N,\epsilon}(x^\epsilon_i(t),t)\phi_x(x^\epsilon_j(t),t)]dt\nonumber\\
&=\int_0^T\sum_{i=1}^Np_i\frac{d}{dt}\phi(x_i^\epsilon(t),t)dt=-\sum_{i=1}^N\phi(x_i(0),0)p_i=-\int_{\mathbb{R}}\phi(x,0)dm^N_0.
\end{align}
On the other hand, combining  the definition  \eqref{m=Gastu} and \eqref{weakfunctional} gives
\begin{align*}
\mathcal{L}(u^{N,\epsilon},\phi)&=\int_0^T\int_{\mathbb{R}}u^{N,\epsilon}[\phi_t-\phi_{txx}]dxdt-\frac{1}{3}\int_0^T\int_{\mathbb{R}}(\partial_xu^{N,\epsilon})^3\phi_{xx}dxdt\\
&\qquad-\frac{1}{3}\int_0^T\int_{\mathbb{R}}(u^{N,\epsilon})^3\phi_{xxx}dxdt+\int_0^T\int_{\mathbb{R}}((u^{N,\epsilon})^3+u^\epsilon(u_x^{N,\epsilon})^2)\phi_xdxdt\\
&=\langle\phi_t,(1-\partial_{xx})u^{N,\epsilon}\rangle+\langle[(u^{N,\epsilon})^2-(\partial_xu^{N,\epsilon})^2](1-\partial_{xx})u^{N,\epsilon},\phi_x\rangle\\
&=\langle m^{N,\epsilon},\phi_t\rangle+\langle  U^N_\epsilon m^{N,\epsilon}, \phi_x\rangle.
\end{align*}
Set
\begin{align}\label{RHSconsistency}
E_{N,\epsilon}:&=\mathcal{L}(u^{N,\epsilon},\phi)+\int_{\mathbb{R}}\phi(x,0)dm^N_0\nonumber\\
&=\langle m^{N,\epsilon}-m^N_\epsilon,\phi_t\rangle+\langle  U^N_\epsilon m^{N,\epsilon}-U^{N,\epsilon} m^N_\epsilon, \phi_x\rangle.
\end{align}
We have the following consistency result.
\begin{proposition}\label{consistency}
We have the following estimate for $E_{N,\epsilon}$ defined by \eqref{RHSconsistency}:
\begin{align}\label{weakconsistence}
|E_{N,\epsilon}|\leq C\epsilon,
\end{align}
where the constant $C$ is independent of $N,\epsilon.$
\end{proposition}

\begin{proof}
By changing of variable and the definition of Schwartz function, we can obtain
\begin{align}\label{eq:schwarz1}
\int_{\mathbb{R}}|x|\rho_\e(x)dx=\int_{\mathbb{R}}|x|\frac{1}{\e}\rho(\frac{x}{\e})dx=\e\int_{\mathbb{R}}|x|\rho(x)dx\leq C_\rho\e,
\end{align}
for some constant $C_\rho.$ 

 Due to $\sum_{i=1}^N|p_i|\leq M_0$ and \eqref{eq:schwarz1}, the first term on the right hand side of \eqref{RHSconsistency} can be estimated as
\begin{align*}
\left|\langle m^{N,\epsilon}-m^N_\epsilon,\phi_t\rangle\right|&=\left|\int_0^T\int_{\mathbb{R}}\sum_{i=1}^Np_i\rho_\epsilon(x-x^\epsilon_i(t))[\phi_t(x,t)-\phi_t(x_i^\epsilon(t),t)]dxdt\right|\\
&\leq \sum_{i=1}^N|p_i|\int_0^T\int_{\mathbb{R}}\rho_\epsilon(x-x^\epsilon_i(t))||\phi_{tx}||_{L^\infty}|x-x_i^\e(t)|dxdt\\
&\leq C_\rho M_0||\phi_{tx}||_{L^\infty}T\epsilon.
\end{align*}
For the second term, by definitions \eqref{mNepisilon} and \eqref{UNepisilon} we can obtain
\begin{align*}
&\quad\langle  U^N_\epsilon m^{N,\epsilon}-U^{N,\epsilon} m^N_\epsilon, \phi_x\rangle\\
&=\sum_{i=1}^Np_i\int_0^T\int_{\mathbb{R}}U^N_\epsilon(x)\rho_\epsilon(x-x_i^\epsilon(t))\phi_x(x,t)dxdt-\sum_{i=1}^Np_i\int_0^TU^{N,\epsilon}(x_i^\epsilon(t))\phi_x(x_i^\epsilon(t),t)dt\\
&=\sum_{i=1}^Np_i\int_0^T\int_{\mathbb{R}}U^N_\epsilon(x)\rho_\epsilon(x-x_i^\epsilon(t))\phi_x(x,t)dxdt\\
&\qquad\qquad\qquad\qquad-\sum_{i=1}^Np_i\int_0^T\int_{\mathbb{R}}U^N_\epsilon(x)\rho_\epsilon(x_i^\epsilon(t)-x)\phi_x(x_i^\epsilon(t),t)dxdt\\
&=\sum_{i=1}^Np_i\int_0^T\int_{\mathbb{R}}U^N_\epsilon(x)\rho_\epsilon(x-x_i^\epsilon(t))[\phi_x(x,t)-\phi_x(x_i^\epsilon(t),t)]dxdt.
\end{align*}
Due to $||U^N_\epsilon||_{L^\infty}\leq \frac{1}{2}M_0^2$, we have
$$\left|\langle  U^N_\epsilon m^{N,\epsilon}-U^{N,\epsilon} m^N_\epsilon, \phi_x\rangle\right|\leq \frac{1}{2}C_\rho M_0^3||\phi_{xx}||_{L^\infty}T\epsilon.$$

This ends the proof.

\end{proof}

\subsection{Convergence theorem}
In this subsection, we prove global existence of $N$-peakon solutions for the mCH equation. 
\begin{theorem}\label{thm:fixedN}
Let $m_0^N(x)$ be given by \eqref{eq:m0N} and $\{x_i^\epsilon(t)\}_{i=1}^N$ is defined by \eqref{approximateODE2} subject to initial data $x^\epsilon_i(0)=c_i$. $u^{N,\epsilon}(x,t)$ is defined by  \eqref{uNepisilon}. Then, the following holds.

(i). There exist $\{x_i(t)\}_{i=1}^N\subset C([0,+\infty))$, such that $x^\epsilon_i\rightarrow x_i$ in $C([0,T])$ as $\epsilon\rightarrow0$ (in subsequence sense) for any $T>0$. Moreover, $x_i(t)$ is global Lipschitz and for a.e. $t>0$, we have
\begin{align}\label{eq:lipschitzxj}
\left|\frac{d}{dt}x_i(t)\right|\leq \frac{1}{2}M_0^2~\textrm{ for }~i=1,\ldots,N.
\end{align}

(ii). Set $u^N(x,t):=\sum_{i=1}^Np_iG(x-x_i(t))$, and we have (in subsequence sense)
\begin{align}\label{eq:l1conver} 
u^{N,\epsilon}\rightarrow u^N, \quad \partial_xu^{N,\epsilon}\rightarrow u_x^N\textrm{ in } L_{loc}^1(\mathbb{R}\times[0,+\infty))\textrm{ as }\e\to0.
\end{align}

(iii). $u^N(x,t)$ is a $N$-peakon solution to \eqref{mCH}.
\end{theorem}
\begin{proof}
(i). Due to $G^\epsilon=G\ast\rho_\e$, we have
$$||G^\epsilon||_{L^\infty}\leq \frac{1}{2}~\textrm{ and }~||G_x^\epsilon||_{L^\infty}\leq \frac{1}{2}.$$
Hence,
\begin{align}\label{eq:uebounded}
||u^{N,\e}||_{L^\infty}\leq \frac{1}{2}M_0~\textrm{ and }~||u_x^{N,\e}||_{L^\infty}\leq \frac{1}{2}M_0,
\end{align}
where $M_0$ is given in \eqref{eq:m0N}.
By Definition \eqref{UNepisilon} and \eqref{eq:uebounded}, we have
\begin{align}\label{eq:Uniformlybounded}
|U^{N,\epsilon}(x,t)|&\leq ||U^N_\epsilon||_{L^\infty}\int_{\mathbb{R}}\rho_\epsilon(x)dx\leq||u^{N,\epsilon}||_{L^\infty}^2+||\partial_xu^{N,\epsilon}||_{L^\infty}^2\nonumber\\
&\leq\frac{1}{4}M_0^2+\frac{1}{4}M_0^2=\frac{1}{2}M_0^2.
\end{align}
Combining \eqref{approximateODE2} and  \eqref{eq:Uniformlybounded}, we have
\begin{align}\label{eq:equicontinuous}
|x^\epsilon_i(t)-x^\epsilon_i(s)|&=\left|\int_s^t\frac{d}{d\tau}x_i^\epsilon(\tau)d\tau\right|=\left|\int_s^tU^{N,\epsilon}(x_i^\epsilon(\tau),\tau)d\tau\right|\nonumber\\
&\leq \int_s^t|U^{N,\epsilon}(x^\epsilon_i(\tau),\tau)|d\tau\leq\frac{1}{2}M_0^2|t-s|.
\end{align}
For each $1\leq i\leq N$, by \eqref{eq:Uniformlybounded} and \eqref{eq:equicontinuous}, we know $\{x_i^\epsilon(t)\}_{\epsilon>0}$ is uniformly (in $\epsilon$) bounded and equi-continuous in $[0,T]$. For any fixed time $T>0$, Arzel\`a-Ascoli Theorem implies that there exists a function $x_i\in C([0,T])$ and a subsequence $\{x_i^{\epsilon_{k}}\}_{k=1}^\infty\subset\{x_i^\epsilon\}_{\epsilon>0}$, such that $x_i^{\epsilon_{k}}\rightarrow x_i$  in $C([0,T])$ as $k\to \infty$. Then, use a diagonalization argument with respect to $T=1,2,\ldots$ and we obtain a subsequence (still denoted as $x_i^\e$)  of $x_i^\e$ such that $x^\epsilon_i\rightarrow x_i$ in $C([0,T])$ as $\epsilon\rightarrow0$ for any $T>0$.  Moreover, by \eqref{eq:equicontinuous}, we have
\begin{align*}
|x_i(t)-x_i(s)|\leq\frac{1}{2}M_0^2|t-s|.
\end{align*}
Hence, $x_i(t)$ is  a global Lipschitz function and \eqref{eq:lipschitzxj} holds.

(ii). Because $u^{N,\epsilon}(x,t)\rightarrow u^N(x,t)$ and $\partial_xu^{N,\epsilon}(x,t)\rightarrow u_x^N(x,t)$ as $\epsilon\rightarrow0$ for a.e. $(x,t)\in\mathbb{R}\times[0,+\infty)$ (for $(x,t)\neq (x_i(t),t)$), then \eqref{eq:l1conver} follows by Lebesgue dominated convergence Theorem.

(iii). Next, we prove that $u^N$ is a weak solution to the mCH equation.

Obviously, we have
$$u^N\in C([0,+\infty);H^1(\mathbb{R}))\cap L^\infty(0,+\infty; W^{1,\infty}(\mathbb{R})).$$
 Similarly as \eqref{consistence0}, for any test function $\phi\in C_c^\infty(\mathbb{R}\times[0,\infty))$ we have
$$\langle m^N_\epsilon,\phi_t\rangle+\langle U^{N,\epsilon} m^N_\epsilon,\phi_x\rangle=-\int_{\mathbb{R}}\phi(x,0)dm_0^N,$$
where $(m^N_\epsilon$, $m^{N,\epsilon})$ is defined by \eqref{mNepisilon} and $(U^N_\epsilon$, $U^{N,\epsilon})$ is defined by \eqref{UNepisilon}.
By the consistency result \eqref{weakconsistence}, we have
\begin{align}\label{Np1}
\mathcal{L}(u^{N,\epsilon},\phi)+\int_{\mathbb{R}}\phi(x,0)dm^N_0\rightarrow0~\textrm{ as }~\epsilon\rightarrow 0,
\end{align}
where
\begin{align}\label{eq:convergenceterms}
\mathcal{L}(u^{N,\epsilon},\phi)&=\int_0^T\int_{\mathbb{R}}u^{N,\epsilon}(\phi_t-\phi_{txx})dxdt-\frac{1}{3}\int_0^T\int_{\mathbb{R}}(\partial_xu^{N,\epsilon})^3\phi_{xx}dxdt\nonumber\\
&\quad-\frac{1}{3}\int_0^T\int_{\mathbb{R}}(u^{N,\epsilon})^3\phi_{xxx}dxdt+\int_0^T\int_{\mathbb{R}}[(u^{N,\epsilon})^3+u^{N,\epsilon}(\partial_xu^{N,\epsilon})^2]\phi_xdxdt.
\end{align}
(Here, $T=T(\phi)$ and $\phi\in C_c^\infty(\mathbb{R}\times[0,T))$.)
We now consider convergence for each term of $\mathcal{L}(u^{N,\epsilon},\phi)$.

For the first term on the right hand side of \eqref{eq:convergenceterms}, using \eqref{eq:l1conver} and  the fact that supp$\{\phi\}$ is compact  we can see
\begin{align*}
\int_0^T\int_{\mathbb{R}}u^{N,\epsilon}(\phi_t-\phi_{txx})dxdt\rightarrow \int_0^T\int_{\mathbb{R}}u^N(\phi_t-\phi_{txx})dxdt~\textrm{ as }~\epsilon\rightarrow0.
\end{align*}
The second term can be estimated as follows
\begin{align*}
&\quad\left|\int_0^T\int_{\mathbb{R}}[(\partial_xu^{N,\epsilon})^3-(u^N_x)^3]\phi_{xx}dxdt\right|\\
&=\left|\int_0^T\int_{\mathbb{R}}(\partial_xu^{N,\epsilon}-u^N_x)[(\partial_xu^{N,\epsilon})^2+(u^N_x)^2+\partial_xu^{N,\epsilon} u^N_x]\phi_{xx}dxdt\right|\\
&\leq \frac{3}{4}M_0^2||\phi_{xx}||_{L^\infty}\int\int_{\mathrm{supp}\{\phi\}}|\partial_xu^{N,\epsilon}-u^N_x|dxdt\rightarrow0~\textrm{ as }~\epsilon\rightarrow0.
\end{align*}
Similarly, we have the following estimates for the rest terms on the right hand side of \eqref{eq:convergenceterms}:
\begin{align*}
\int_0^T\int_{\mathbb{R}}[(u^{N,\epsilon})^3-(u^N)^3]\phi_{xxx}dxdt\rightarrow0~\textrm{ as }~\epsilon\rightarrow0,\\
\int_0^T\int_{\mathbb{R}}[(u^{N,\epsilon})^3-(u^N)^3]\phi_{x}dxdt\rightarrow0~\textrm{ as }~\epsilon\rightarrow0,
\end{align*}
and
\begin{align*}
&\quad\int_0^T\int_{\mathbb{R}}[u^{N,\epsilon}(\partial_xu^{N,\epsilon})^2-u^N(u^N_x)^2]\phi_xdxdt\\
&=\int_0^T\int_{\mathbb{R}}[(u^{N,\epsilon}-u^N)(\partial_xu^{N,\epsilon})^2+u^N(\partial_xu^{N,\epsilon}+u^N_x)(\partial_xu^{N,\epsilon}-u^N_x)]\phi_xdxdt\\
&\rightarrow0~\textrm{ as }~\epsilon\rightarrow0.
\end{align*}
Hence, the above estimates shows that for any test function $\phi\in C_c^\infty(\mathbb{R}\times[0,\infty))$
\begin{align}\label{Np2}
\mathcal{L}(u^{N,\epsilon},\phi)\rightarrow\mathcal{L}(u^{N},\phi)~\textrm{ as }~\epsilon\rightarrow 0.
\end{align}
Therefore, combining \eqref{Np1} and \eqref{Np2} gives
$$\mathcal{L}(u^{N},\phi)+\int_{\mathbb{R}}\phi(x,0)dm_0^N=0,$$
which implies that $u^N(x,t)$ is a $N$-peakon solution to the mCH equation with initial date $m^N_0(x)$.
\end{proof}

\subsection{A Limiting system of ODEs as $\e\to0$}\label{secunique}
In this section, we derive a system of ODEs to describe  $N$-peakon solutions by letting $\epsilon\rightarrow0$ in \eqref{approximateODE2}. 
First, we give an important lemma.
\begin{lemma}\label{convergencetoNpeakon}
The following equality holds
$$\lim_{\epsilon\rightarrow 0}(\rho_\epsilon\ast (G_x^\epsilon)^2)(0)=\frac{1}{12}.$$
\end{lemma}
\begin{proof}
Set $F(y)=\int_{-\infty}^y\rho(x)dx.$ Because $\rho$ is an even function, we have
$$F(-y)=\int_{-\infty}^{-y}\rho(x)dx=\int_y^{\infty}\rho(x)dx.$$
Therefore, 
\begin{align}\label{suma1}
F(y)+F(-y)=\int_{-\infty}^y\rho(x)dx+\int_y^{\infty}\rho(x)dx=1.
\end{align}
 Furthermore, we have
$$F(+\infty)=1,\quad F(-\infty)=0.$$
Due to $\rho_\epsilon(x)=\rho_\epsilon(-x)$, we can obtain
\begin{align*}
I_\epsilon:=(\rho_\epsilon\ast (G_x^\epsilon)^2)(0)&=\int_{\mathbb{R}}\rho_\epsilon(y)\left(\int_{\mathbb{R}}\frac{1}{2}e^{-|x-y|}\rho'_\epsilon(x)dx\right)^2dy\\
&=\frac{1}{4}\int_{\mathbb{R}}\rho(y)\left(\frac{1}{\epsilon}\int_{-\infty}^ye^{\epsilon(x-y)}\rho'(x)dx+\frac{1}{\epsilon}\int^{\infty}_ye^{\epsilon(y-x)}\rho'(x)dx\right)^2dy\\
&=\frac{1}{4}\int_{\mathbb{R}}\rho(y)\left(\int_{-\infty}^ye^{-\epsilon|x-y|}\rho(x)dx-\int^{\infty}_ye^{-\epsilon|x-y|}\rho(x)dx\right)^2dy.
\end{align*}
Then, by using Lebesgue dominated convergence Theorem and \eqref{suma1} we have
\begin{align*}
\lim_{\epsilon\rightarrow0}I_\epsilon&=\frac{1}{4}\int_{\mathbb{R}}\rho(y)\left(\int_{-\infty}^y\rho(x)dx-\int^{\infty}_y\rho(x)dx\right)^2dy\\
&=\frac{1}{4}\int_{\mathbb{R}}\rho(y)(F(y)-F(-y))^2dy=\frac{1}{4}\int_{-\infty}^\infty F'(y)(1-2F(y))^2dy\\
&=\frac{1}{4}\int_{-\infty}^\infty F'(y)-2(F^2(y))'+\frac{4}{3}(F^3(y))'dy\\
&=\frac{1}{4}\left(F(+\infty)-2F^2(+\infty)+\frac{4}{3}F^3(+\infty)\right)=\frac{1}{12}.
\end{align*}

\end{proof}
\begin{remark}
The above limit is independent of the mollifier $\rho$ and intrinsic to the mCH equation \eqref{mCH}. Consider one peakon solution $pG(x-x(t))$. To obtain the correct speed for $x(t)$, the right value for $G_x^2$ at $0$ is the limit obtained by Lemma \ref{convergencetoNpeakon}:
$$(G_x^2)(0)=\frac{1}{12}.$$
By the jump condition for piecewise smooth weak solutions to \eqref{mCH} in \cite[Equation (2.2)]{GaoLiLiu}, the speed for $x(t)$ should be
\[
\frac{dx(t)}{dt}=G^2(0)-\frac{1}{3}[G_x^2(0+)+G_x(0+)G_x(0-)+G_x^2(0-)],
\]
implying that the correct value of $G_x^2$ at $0$ is
\[
\frac{1}{3}[G_x^2(0+)+G_x(0+)G_x(0-)+G_x^2(0-)]=\frac{1}{12},
\]
which agrees with the limit obtained by Lemma \ref{convergencetoNpeakon}.
This is different from the precise representative of the BV function $G_x^2$ at the discontinuous point $0$
\[
\frac{1}{2}[G_x^2(0-)+G_x^2(0+)]=\frac{1}{4}.
\]

\end{remark}
Next, we use Lemma \ref{convergencetoNpeakon} to obtain the system of ODEs to describe $N$-peakon solutions by letting $\epsilon\rightarrow0$ in \eqref{approximateODE2}.

\begin{proposition}\label{discontinue}
For any constants $\{p_i\}_{i=1}^N$, $\{x_i\}_{i=1}^N\subset\mathbb{R}$ (note that $x_i$ are fixed compared with $x_i^\e(t)$ in \eqref{uNepisilon}),  denote $\mathcal{N}_{i1}:=\{1\leq j\leq N:x_j\neq x_i\}$ and $\mathcal{N}_{i2}:=\{1\leq j\leq N:x_j=x_i\}$ for $1\leq i\leq N$. Set
$$u^{N,\epsilon}(x):=\sum_{j=1}^Np_jG^\epsilon(x-x_j),$$
and
$$U^\epsilon(x):=[\rho_\epsilon\ast (u^{N,\epsilon})^2](x)-[\rho_\epsilon\ast (u_x^{N,\epsilon})^2](x).$$
(Note that $x_i$ are constants in $U^\e(x)$ comparing with $U^{N,\e}(x,t)$ defined by \eqref{UNepisilon}.) Then we have
\begin{align}\label{Npeakonequation}
\lim_{\epsilon\rightarrow0}U^\epsilon(x_i)&=\left(\sum_{j=1}^Np_jG(x_i-x_j)\right)^2-\left(\sum_{j\in \mathcal{N}_{i1}}p_j G_x(x_i-x_j)\right)^2-\frac{1}{12}\left(\sum_{k\in \mathcal{N}_{i2}}p_k\right)^2.
\end{align}
\end{proposition}
\begin{proof}
Because $\sum_{j=1}^Np_jG(x-x_j)$ is continuous, we have
\begin{align}\label{firstterm}
\lim_{\epsilon\rightarrow0}\rho_\epsilon\ast (u^{N,\epsilon})^2(x_i)=\left(\sum_{j=1}^Np_jG(x_i-x_j)\right)^2.
\end{align}
Next we estimate the second term $[\rho_\epsilon\ast (u_x^{N,\epsilon})^2](x_i)$ in $U^\epsilon(x_i)$.
We have
\begin{align}\label{secondterm0}
&\quad(u_x^{N,\epsilon})^2(x)=\left(\sum_{j\in \mathcal{N}_{i1}}p_j G_x^\epsilon(x-x_j)\right)^2+2\sum_{j\in \mathcal{N}_{i1},k\in \mathcal{N}_{i2}}p_j G_x^\epsilon(x-x_j)p_kG^\epsilon_x(x-x_k)\nonumber\\
&\quad+\left(\sum_{k\in \mathcal{N}_{i2}}p_k G_x^\epsilon(x-x_k)\right)^2=:F^\epsilon_1(x)+F^\epsilon_2(x)+F^\epsilon_3(x).
\end{align}
Because  $G_x(x)$ is continuous at $x_i-x_j$, we have the following estimate for $F^\epsilon_1$
\begin{align}\label{1111}
\lim_{\epsilon\rightarrow0}(\rho_\epsilon\ast F^\epsilon_1)(x_i)=\left(\sum_{j\in \mathcal{N}_{i1}}p_j G_x(x_i-x_j)\right)^2.
\end{align}
Because $G$ and $\rho_\epsilon$ are even functions, we know $G^\epsilon_x$ is an odd function. Next, consider  the second term $F^\epsilon_2$ on the right hand side of \eqref{secondterm0}. Due to $x_k=x_i$ for $k\in \mathcal{N}_{i2}$, we have

\begin{align}\label{222}
(\rho_\epsilon\ast F^\epsilon_2)(x_i)=&2\sum_{j\in \mathcal{N}_{i1},k\in \mathcal{N}_{i2}}p_jp_k\int_{\mathbb{R}}\rho_\epsilon(x_i-y)G_x^\epsilon(y-x_j)G_x^\epsilon(y-x_i)dy\nonumber\\
=&2\sum_{j\in \mathcal{N}_{i1},k\in \mathcal{N}_{i2}}p_jp_k\int_0^\infty\rho_\epsilon(y)G_x^\epsilon(-y)\nonumber\\
&\times\left(\int_{\mathbb{R}}\Big[G_x(x_i-x_j-y-x)-G_x(x_i-x_j+y-x)\Big]\rho_\epsilon(x)dx\right)dy\nonumber\\
\leq &2\sum_{j\in \mathcal{N}_{i1},k\in \mathcal{N}_{i2}}p_jp_k\int_0^{\sqrt{\epsilon}}\rho_\epsilon(y)G_x^\epsilon(-y)\nonumber\\
& \times\left(\int_{-\sqrt{\epsilon}}^{\sqrt{\epsilon}}\Big|G_x(x_i-x_j-y-x)-G_x(x_i-x_j+y-x)\Big|\rho_\epsilon(x)dx\right)dy\nonumber\\
&+ 3\sum_{j\in \mathcal{N}_{i1},k\in \mathcal{N}_{i2}}p_jp_k\int_{\sqrt{\epsilon}}^\infty\rho_\epsilon(y)dy=:I^\e_1+I^\e_2.
\end{align}
Due to $x_j\neq x_i$ for $j\in \mathcal{N}_{i1}$, we can choose $\epsilon$ small enough such that
$$(x_i-x_j-y-x)(x_i-x_j+y-x)>0,\textrm{ for }|x|,|y|<\sqrt{\epsilon}.$$
Hence,
$$|G_x(x_i-x_j-y-x)-G_x(x_i-x_j+y-x)|\leq\frac{1}{2}|2y|<\sqrt{\epsilon}.$$
Putting the above estimate into $I^\e_1$ gives
\begin{align}\label{eq:221}
I_1^\e&=2\sum_{j\in \mathcal{N}_{i1},k\in \mathcal{N}_{i2}}p_jp_k\int_0^{\sqrt{\epsilon}}\rho_\epsilon(y)G_x^\epsilon(-y)\nonumber\\
&\qquad \times\left(\int_{-\sqrt{\epsilon}}^{\sqrt{\epsilon}}\Big|G_x(x_i-x_j-y-x)-G_x(x_i-x_j+y-x)\Big|\rho_\epsilon(x)dx\right)dy\nonumber\\
&\leq \sum_{j\in \mathcal{N}_{i1},k\in \mathcal{N}_{i2}}|p_jp_k|\cdot\sqrt{\epsilon}\rightarrow 0~\textrm{ as }~\epsilon\rightarrow0.
\end{align}
For $I_2^\epsilon$, changing variable gives
\begin{align}\label{eq:222}
I_2^\e&= 3\sum_{j\in \mathcal{N}_{i1},k\in \mathcal{N}_{i2}}p_jp_k\int_{\sqrt{\epsilon}}^\infty\rho_\epsilon(y)dy\nonumber\\
&= 3\sum_{j\in \mathcal{N}_{i1},k\in \mathcal{N}_{i2}}p_jp_k\int_{\frac{1}{\sqrt{\epsilon}}}^\infty\rho(y)dy\rightarrow 0~\textrm{ as }~\epsilon\rightarrow0.
\end{align}
Combining \eqref{222}, \eqref{eq:221}, and \eqref{eq:222}, we have
\begin{align}\label{2222}
\lim_{\e\to0}|(\rho_\epsilon\ast F^\epsilon_2)(x_i)|=0.
\end{align}
For $F^\epsilon_3$ in \eqref{secondterm0},  using Lemma \ref{convergencetoNpeakon} we can obtain
\begin{align}\label{3333}
\lim_{\epsilon\rightarrow0}(\rho_\epsilon\ast F_3^\epsilon)(x_i)
&=\lim_{\epsilon\rightarrow0}\int_{\mathbb{R}}\rho_\epsilon(x_i-y)\left(\sum_{k\in \mathcal{N}_{i2}}p_k \int_{\mathbb{R}}G(y-x_k-x)\rho_\epsilon(x)dx\right)^2dy\nonumber\\
&=\left(\sum_{k\in \mathcal{N}_{i2}}p_k\right)^2\lim_{\epsilon\rightarrow0}\int_{\mathbb{R}}\rho_\epsilon(y)\bigg(\int_{\mathbb{R}}G(y-x)\rho_\epsilon(x)dx\bigg)^2dy\nonumber\\
&=\left(\sum_{k\in \mathcal{N}_{i2}}p_k\right)^2\lim_{\epsilon\rightarrow0}\left[(G_x^{\epsilon})^2\ast\rho_\epsilon\right](0)\nonumber\\
&=\frac{1}{12}\left(\sum_{k\in \mathcal{N}_{i2}}p_k\right)^2,
\end{align}
where we used $x_i=x_k$ for $k\in \mathcal{N}_{i2}$ in the second step.
Finally, combining \eqref{1111}, \eqref{2222} and \eqref{3333} gives
\begin{align}\label{secondterm}
\lim_{\epsilon\rightarrow0}[\rho_\epsilon\ast (u_x^{N,\epsilon})^2](x_i)=\frac{1}{12}\left(\sum_{k\in \mathcal{N}_{i2}}p_k\right)^2+\left(\sum_{j\in\mathcal{N}_{i1}}p_j G_x(x_i-x_j)\right)^2.
\end{align}
Combining \eqref{firstterm} and \eqref{secondterm} gives  \eqref{Npeakonequation}.

\end{proof}

\begin{remark}[System of ODEs]\label{ODEsfor trajiecories}
From Proposition \ref{discontinue}, we give a system of ODEs to describe  $N$-peakon solution $u^N(x,t)=\sum_{i=1}^Np_iG(x-x_i(t)).$
For $1\leq i\leq N$, set $\mathcal{N}_{i1}(t)=\{1\leq j\leq N:x_j(t)\neq x_i(t)\}$ and $\mathcal{N}_{i2}(t)=\{1\leq j\leq N:x_j(t)=x_i(t)\}$. The  system of ODEs is given by, $1\leq i\leq N,$
\begin{footnotesize}
\begin{align}\label{NpeakontrajectoriesODE}
\frac{d}{dt}x_i(t)=\left(\sum_{j=1}^Np_jG(x_i(t)-x_j(t))\right)^2-\left(\sum_{j\in \mathcal{N}_{i1}(t)}p_j G_x(x_i(t)-x_j(t))\right)^2-\frac{1}{12}\left(\sum_{k\in \mathcal{N}_{i2}(t)}p_k\right)^2.
\end{align}
\end{footnotesize}
Before the collisions of peakons, we can deduce \eqref{eq:Npeakonmonoticity0} from \eqref{NpeakontrajectoriesODE}.
\end{remark}
\begin{remark}[nonuniqueness]\label{rmk:nonuniqueness}
Consider the initial two peakons $p_1\delta(x-x_1(0))+p_2\delta(x-x_2(0))$ with $x_1(0)<x_2(0)$ and $0<p_2<p_1$. Before collision, the evolution system for $x_1(t)$ and $x_2(t)$ is given by
\begin{gather}\label{eq:twobefore}
\left\{
\begin{split}
\frac{d}{dt}x_1(t)=\frac{1}{6}p_1^2+\frac{1}{2}p_1p_2e^{x_1(t)-x_2(t)},\\
\frac{d}{dt}x_2(t)=\frac{1}{6}p_2^2+\frac{1}{2}p_1p_2e^{x_1(t)-x_2(t)}.
\end{split}
\right.
\end{gather}
This system is the same as \eqref{eq:Npeakonmonoticity0} for $N=2$.
The relative speed of the first peakon with respect to the second one is $\frac{1}{6}(p_1^2-p_2^2)>0$. Hence, they will collide at finite time $T_*=\frac{6(x_2(0)-x_1(0))}{p_1^2-p_2^2}$. 
When $t>T_*$, if we assume the two peakons sticky together, according to \eqref{NpeakontrajectoriesODE} the evolution equation is given by
\begin{align}\label{eq:twopeakonsticky}
\frac{d}{dt}x_i(t)=\frac{1}{6}(p_1+p_2)^2,~t>T_*,~~i=1,2.
\end{align}
The peakon solution $u(x,t)=p_1G(x-x_1(t))+p_2G(x-x_2(t))$ constructed by \eqref{eq:twobefore} and \eqref{eq:twopeakonsticky} corresponds to the sticky peakon weak solution given by \cite{GaoLiu}. In next section, we will prove that when $N=2$, the limiting peakon solution (for $t>T_*$) given by Theorem \ref{thm:fixedN} also corresponds to $u(x,t)$, which means it is a sticky peakon weak solution.

If we assume the two peakons cross each other when $t>T_*$ (still with amplitudes $p_1$, $p_2$), then according to \eqref{NpeakontrajectoriesODE}, the evolution equation for $x_1(t)$ and $x_2(t)$ is given by
\begin{gather}\label{eq:twoafter}
\left\{
\begin{split}
\frac{d}{dt}x_1(t)=\frac{1}{6}p_1^2+\frac{1}{2}p_1p_2e^{x_2(t)-x_1(t)},~~t>T_*,\\
\frac{d}{dt}x_2(t)=\frac{1}{6}p_2^2+\frac{1}{2}p_1p_2e^{x_2(t)-x_1(t)},~~t>T_*.
\end{split}
\right.
\end{gather}
This system is different with \eqref{eq:Npeakonmonoticity0} and one can easily check that, $\bar{u}(x,t)=p_1G(x-\bar{x}_1(t))+p_2G(x-\bar{x}_2(t))$ constructed by \eqref{eq:twobefore} and \eqref{eq:twoafter} is a weak solution while $\tilde{u}(x,t)=p_1G(x-\tilde{x}_1(t))+p_2G(x-\tilde{x}_2(t))$ constructed by  \eqref{eq:Npeakonmonoticity0} is not a weak solution for $t>T_*$.

Both $u(x,t)$ and $\bar{u}(x,t)$ are global two peakon solutions, which proves nonuniqueness of weak solutions to the mCH equation.
 This nonuniqueness example can also be found in \cite[Proposition 4.4]{GaoLiu}.
 
The above example also shows that after collision peakons can merge into one or cross each other. Moreover, if we view $T_*$ as the start point with one peakon, then the above example shows the scattering of one peakon. This indicates all the situation mentioned in question (iii) in Introduction.
\end{remark}

At the end of this section, we give a useful proposition.
\begin{proposition}\label{pro:Npeakon}
	Let $x_i(t)$, $1\leq i\leq N$, be $N$ Lipschitz functions in $[0,T)$ with $x_1(t)< x_2(t)<\cdots<x_N(t)$ and $p_1,\cdots,p_N$ are $N$ non-zero constants. Then, $u^N(x,t):=\sum_{i=1}^Np_iG(x-x_i(t))$ is a weak solution to the mCH equation if and only if $x_i(t)$ satisfies \eqref{eq:Npeakonmonoticity0}.
\end{proposition}
\begin{proof}
	Obviously, we have
	$$u^N\in C([0,T);H^1(\mathbb{R}))\cap L^\infty(0,T; W^{1,\infty}(\mathbb{R})).$$
	In the following proof we denote $u:=u^N$. For any   test function $\phi\in C_c^\infty(\mathbb{R}\times[0,T))$,  let
	\begin{align}\label{weakfunctional2}
	\mathcal{L}(u,\phi)&=\int_0^{T}\int_{\mathbb{R}}u(\phi_t-\phi_{txx})dxdt-\int_0^{T}\int_{\mathbb{R}}\left[\frac{1}{3}(u^3_x\phi_{xx}+u^3\phi_{xxx})-(u^3+uu_x^2)\phi_x\right]dxdt\nonumber\\
	&=:I_1+I_2.
	\end{align}
	Denote $x_0:=-\infty$, $x_{N+1}:=+\infty$ and $p_0=p_{N+1}=0$.
	By integration by parts for space variable $x$, we calculate $I_1$ as
	\begin{align}\label{I11}
	I_1&=\int_0^{T}\int_{\mathbb{R}}u(\phi_t-\phi_{txx})dxdt=\sum_{i=0}^{N}\int_0^{T}\int_{x_i}^{x_{i+1}}u(\phi_t-\phi_{txx})dxdt\nonumber\\
	&=\sum_{i=0}^{N}\int_0^{T}\int_{x_i}^{x_{i+1}}\left(\frac{1}{2}\sum_{j\leq i}p_je^{x_j-x}+\frac{1}{2}\sum_{j>i}p_je^{x-x_j}\right)(\phi_t-\phi_{txx})dxdt\nonumber\\
	&=\int_0^{T}\sum_{i=1}^Np_i\phi_t(x_i(t),t)dt.
	\end{align}
	Similarly, for $I_2$ we have
	\begin{align}\label{I21}
	I_2&=-\int_0^{T}\int_{\mathbb{R}}\left[\frac{1}{3}(u^3_x\phi_{xx}+u^3\phi_{xxx})-(u^3+uu_x^2)\phi_x\right]dxdt\nonumber\\
	&=\int_0^{T}\sum_{i=1}^Np_i\phi_x(x_i(t))\left(\frac{1}{6}p_i^2+\frac{1}{2}\sum_{j<i}p_ip_je^{x_j-x_i}+\frac{1}{2}\sum_{j>i}p_ip_je^{x_i-x_j}\right.\nonumber\\
	&\left.\qquad\qquad\qquad\qquad\qquad\qquad+\sum_{1\leq m<i<n\leq N}p_mp_ne^{x_m-x_n}\right)dt\nonumber\\
	&=\int_0^{T}\sum_{i=1}^Np_i\phi_x(x_i(t))F(t)dt.
	\end{align}
	where 
	$$F(t):=\frac{1}{6}p_i^2+\frac{1}{2}\sum_{j<i}p_ip_je^{x_j-x_i}+\frac{1}{2}\sum_{j>i}p_ip_je^{x_i-x_j}+\sum_{1\leq m<i<n\leq N}p_mp_ne^{x_m-x_n}.$$
	Combining  \eqref{weakfunctional2}, \eqref{I11} and \eqref{I21} gives
	\begin{align}\label{useful}
	\mathcal{L}(u,\phi)&=\sum_{i=1}^Np_i\int_0^{T}\frac{d}{dt}\phi(x_i(t),t)dt+\int_0^{T}\sum_{i=1}^Np_i\phi_x(x_i(t))\left(F(t)-\frac{d}{dt}x_i(t)\right)dt\nonumber\\
	&=-\int_{\mathbb{R}}\phi(x,0)dm_0^N+\int_0^{T}\sum_{i=1}^Np_i\phi_x(x_i(t))\left(F(t)-\frac{d}{dt}x_i(t)\right)dt.
	\end{align}
	By Definition \ref{weaksolution} we know $u^N$ is a weak solution if and only if $\frac{d}{dt}x_i(t)=F(t)$, which is \eqref{eq:Npeakonmonoticity0}.
	
\end{proof}
\begin{remark}\label{rmk:rightequation}
Form Remark \ref{rmk:nonuniqueness} and Proposition \ref{pro:Npeakon}, we know that solutions to \eqref{eq:Npeakonmonoticity0} can not be used to construct peakon weak solutions after $t>T_*$. Because $x_1(t)>x_2(t)$ when $t>T_*$, \eqref{eq:twoafter} is the right evolution equation for $x_i(t)$, $i=1,2$. 

Proposition \ref{pro:Npeakon} also implies the uniqueness of the limiting trajectories $x_i(t)$ before collision.
\end{remark}

\section{Limiting peakon solutions as $\epsilon\to 0$} \label{sec:stickylim}
In this section, we analysis peakon solutions given by the dispersive regularization.

\subsection{No collisions for the regularized system}\label{subsec:nocollision}
In this subsection, we show that trajectories $\{x_i^\e(t)\}_{i=1}^N$ obtained by \eqref{approximateODE2} will never collide.
Define
\begin{gather}\label{eq:defoffie}
f_1^{\e}(x):=\frac{1}{2}\int_0^{\infty}\rho_{\epsilon}(x-y)e^{-y}dy~\textrm{ and }~f_2^{\e}(x):=\frac{1}{2}\int_{-\infty}^{0}\rho_{\epsilon}(x-y)e^{y}dy.
\end{gather}
Changing variable gives
\begin{gather}\label{eq:defoffie1}
f_1^{\e}(x)=\frac{1}{2}\int_{-\infty}^x\rho_{\epsilon}(y)e^{y-x}dy~\textrm{ and }~f_2^{\e}(x)=\frac{1}{2}\int_{x}^{\infty}\rho_{\epsilon}(y)e^{x-y}dy.
\end{gather}
It is easy to see that both $f_1^{\e}, f_2^{\e}\in C^{\infty}(\mathbb{R})$ and we have the following lemma.
\begin{lemma}\label{lmm:noncollision}
Let $C_0:=||\rho||_{L^\infty}$.  Then, the following properties for $f_i^\e$ ($i=1,2$) hold:

$\mathrm{(i)}$
\begin{align}\label{eq:fiproperties1}
f_2^{\e}(x)=f_1^{\e}(-x),~~G^{\e}(x)=f_1^{\e}+f_2^{\e},~\textrm{ and }~G_x^{\e}(x)=-f_1^{\e}(x)+f_2^{\e}(x).
\end{align}

$\mathrm{(ii)}$
\begin{gather}\label{eq:fiproperties2}
||f_1^\e||_{L^\infty},||f_2^\e||_{L^\infty}\leq \frac{1}{2},~\textrm{ and }~||\partial_xf_1^\e||_{L^\infty},||\partial_xf_2^\e||_{L^\infty}\leq \frac{C_0}{2\e}+\frac{1}{2}.
\end{gather}

\end{lemma}
\begin{proof}
(i). The first two equalities in \eqref{eq:fiproperties1} can be easily proved. For the third one, taking derivative of \eqref{eq:defoffie1} gives
\begin{align}\label{eq:derivativeoffi}
\partial_xf_1^\e(x)=\frac{1}{2}\rho_\e(x)-f_1^\e(x),~\textrm{ and }~\partial_xf_2^\e(x)=-\frac{1}{2}\rho_\e(x)+f_2^\e(x).
\end{align}
Hence, we have $G_x^{\e}(x)=-f_1^{\e}(x)+f_2^{\e}(x).$

(ii). By Definition \eqref{eq:defoffie}, we can obtain 
$$||f_1^\e||_{L^\infty},||f_2^\e||_{L^\infty}\leq \frac{1}{2}.$$
Due to \eqref{eq:derivativeoffi} and  $C_0=||\rho||_{L^\infty}$, we have
$$||\partial_xf_1^\e||_{L^\infty},||\partial_xf_2^\e||_{L^\infty}\leq \frac{C_0}{2\e}+\frac{1}{2}.$$

\end{proof}

\begin{theorem}\label{thm:nocollisionepsilon}
Let $\{x_i^\e(t)\}_{i=1}^N$ be a solution to \eqref{approximateODE2} subject to $x_i^\e(0)=c_i$, $i=1,\ldots, N$ and $\sum_{i=1}^N|p_i|\leq M_0$ for some constant $M_0$.   If $c_1<c_2<\cdots<c_N$, then $x^\e_1(t)<x^\e_2(t)<\cdots<x_N^\e(t)$ for all $t>0$. 
\end{theorem}
\begin{proof}

If collision between $\{x_i^\e\}_{i=1}^N$ happens, we assume the first collision is between  $x_k^\e$ and $x_{k+1}^\e$ for some $1\leq k\leq N-1$ at time $T_*>0$. Our target is to prove $T_*=+\infty.$

By \eqref{uNepisilon} and \eqref{eq:fiproperties1}, we have
$$u^{N,\e}(x,t)=\sum_{i=1}^Np_iG^\e(x-x_i^\e)=\sum_{i=1}^Np_i\left(f_1^\e(x-x_i^\e)+f_2^\e(x-x_i^\e)\right),$$
and
$$u_x^{N,\e}(x,t)=\sum_{i=1}^Np_iG_x^\e(x-x_i^\e)=\sum_{i=1}^Np_i\left(-f_1^\e(x-x_i^\e)+f_2^\e(x-x_i^\e)\right).$$
Hence, we obtain
\begin{align*}
U_\e^N(x,t)&=(u^{N,\e}+u^{N,\e}_x)(u^{N,\e}-u^{N,\e}_x)=4\left(\sum_{i=1}^Np_if_2^\e(x-x_i^\e)\right)\left(\sum_{i=1}^Np_if_1^\e(x-x_i^\e)\right).
\end{align*}
From \eqref{approximateODE2}, we have
\begin{gather}
\frac{d}{dt}x^\e_k=\big[\rho_{\e}*U_\e^N\big](x^\e_k)~\textrm{ and }~ \frac{d}{dt}x^\e_{k+1}=\big[\rho_{\e}*U_\e^N\big](x^\e_{k+1}).
\end{gather}
For $t< T_*$, taking the difference gives
\begin{align*}
&\frac{d}{dt}(x_{k+1}^\e-x^\e_k)\nonumber\\
=&4\int_{\mathbb{R}}\rho_\e(y)\left(\sum_{i=1}^Np_if_2^\e(x_{k+1}^\e-y-x_i^\e)\right)\left(\sum_{i=1}^Np_if_1^\e(x_{k+1}^\e-y-x_i^\e)\right)dy\\
&-4\int_{\mathbb{R}}\rho_\e(y)\left(\sum_{i=1}^Np_if_2^\e(x_{k}^\e-y-x_i^\e)\right)\left(\sum_{i=1}^Np_if_1^\e(x_{k}^\e-y-x_i^\e)\right)dy\\
=&4\int_{\mathbb{R}}\rho_\e(y)\left(\sum_{i=1}^Np_if_2^\e(x_{k+1}^\e-y-x_i^\e)\right)\sum_{i=1}^Np_i\left(f_1^\e(x_{k+1}^\e-y-x_i^\e)-f_1^\e(x_{k}^\e-y-x_i^\e)\right)dy\\
&+4\int_{\mathbb{R}}\rho_\e(y)\left(\sum_{i=1}^Np_if_1^\e(x_{k}^\e-y-x_i^\e)\right)\sum_{i=1}^Np_i\left(f_2^\e(x_{k+1}^\e-y-x_i^\e)-f_2^\e(x_{k}^\e-y-x_i^\e)\right)dy.
\end{align*}
Combining \eqref{eq:fiproperties1} and \eqref{eq:fiproperties2} gives
\begin{align}
\left|\frac{d}{dt}(x_{k+1}^\e-x^\e_k)\right|\leq&2M_0^2||\partial_xf_1^\e||_{L^\infty}(x_{k+1}^\e-x_k^\e)+2M_0^2||\partial_xf_2^\e||_{L^\infty}(x_{k+1}^\e-x_k^\e)\nonumber\\
\leq &C_\e(x_{k+1}^\e-x_k^\e),~~t< T_*,
\end{align}
where 
$$C_\e=M_0^2\left(\frac{C_0}{\e}+1\right).$$
Hence, for $t<T_*$
\begin{align}
-C_\e(x_{k+1}^\e-x_k^\e)\leq \frac{d}{dt}(x_{k+1}^\e-x^\e_{k})\leq  C_\e(x_{k+1}^\e-x_k^\e),
\end{align}
which implies
$$0<(c_{k+1}-c_k)e^{-C_\e t}\leq x_{k+1}^\e(t)-x_k^\e(t)~\textrm{ for }~t< T_*.$$
By our assumption about $T_*$, we know $T_*=+\infty.$ Hence, we have $x^\e_1(t)<x^\e_2(t)<\cdots<x_N^\e(t)$ for all $t>0$.

\end{proof}

\begin{remark}\label{rmk:nocross}
Let  $u^N(x,t)=\sum_{i=1}^NG(x-x_i(t))$ be a $N$-peakon solution to the mCH equation obtained by Theorem \ref{thm:fixedN}. From Theorem \ref{thm:nocollisionepsilon}, we have
\begin{align}\label{eq:nocross}
x_1(t)\leq x_2(t)\leq \cdots\leq x_N(t).
\end{align}
This result shows that the limit solution allows no cross between peakons. 

\end{remark}

\subsection{Two peakon solutions} \label{sec:twopeakon}
As mentioned in Introduction, the sticky peakon solutions given in \cite{GaoLiu} also satisfy \eqref{eq:nocross}. 
In this subsection, when $N=2$, we show that the limiting $N$-peakon solutions given in Theorem \ref{thm:fixedN} agree with sticky peakon solutions (see $u(x,t)$ in Remark \ref{rmk:nonuniqueness}). Due to Proposition \ref{pro:Npeakon}, the cases with no collisions are easy to verify.

Consider the case with a collision for $N=2$. When $p_1^2>p_2^2$ and $x_1(0)=c_1<c_2=x_2(0)$, the equations for $x_1(t)$ and $x_2(t)$ before collisions are given by
\begin{gather}
\left\{\begin{split}
\frac{d}{dt}x_1(t)=\frac{1}{6}p_1^2+\frac{1}{2}e^{x_1(t)-x_2(t)},\\
\frac{d}{dt}x_2(t)=\frac{1}{6}p_2^2+\frac{1}{2}e^{x_1(t)-x_2(t)}.
\end{split}
\right.
\end{gather}
The two peakons collide at $T_*=\frac{6(c_2-c_1)}{p^2_1-p_2^2}$. Next, we prove the following theorem.
\begin{theorem}\label{thm:twosticky}
Assume $N=2$ and $m_0^N(x)=p_1\delta(x-c_1)+p_2\delta(x-c_2)$ with $p_1^2>p_2^2$ and $c_1<c_2$. Then, the peakon solution $u^N(x,t)=p_1G(x-x_1(t))+p_2G(x-x_2(t))$ obtained in Theorem \ref{thm:fixedN} is a sticky peakon solution, which means
\begin{align}\label{eq:collision}
x_1(t)=x_2(t)~\textrm{ for }~t\geq T_*:=\frac{6(c_2-c_1)}{p_1^2-p_2^2}.
\end{align}
\end{theorem}

To prove Theorem \ref{thm:twosticky}, we first consider \eqref{approximateODE2} for $N=2$. Denote $S_\e(t):=x^\e_2(t)-x^\e_1(t)>0$. By the fact that $f^\e_1(-x)=f^\e_2(x)$, we find that 
\begin{align}\label{eq:derivativex1}
\frac{d}{dt}x^\e_1&=4\int_{-\infty}^{\infty}\rho_{\e}(y)\big[p_1f_2^{\e}(-y)+p_2f_2(-S_\e-y)\big]\big[p_1f_1^{\e}(-y)+p_2f_1^{\e}(-S_\e-y)\big] dy\nonumber\\
&=4\int_{-\infty}^{\infty}\rho_{\e}(y)\big[p_1f_1^{\e}(y)+p_2f^\e_1(S_\e+y)\big]\big[p_1f_2^{\e}(y)+p_2f_2^{\e}(S_\e+y)\big] dy.
\end{align}
By changing of variables $y\to -y$ and using the fact that $\rho_{\e}$ is even, we obtain 
\begin{align}\label{eq:derivativex2}
\frac{d}{dt}x^\e_2&=4\int_{-\infty}^{\infty}\rho_{\e}(y)\big[p_1f_2^{\e}(S_\e-y)+p_2f_2(-y)\big]\big[p_1f_1^{\e}(S_\e-y)+p_2f_1^{\e}(-y)\big] dy\nonumber\\
&=4\int_{-\infty}^{\infty}\rho_{\e}(y)\big[p_1f_2^{\e}(S_\e+y)+p_2f_2^\e(y)\big]\big[p_1f_1^{\e}(S_\e+y)+p_2f_1^{\e}(y)\big] dy
\end{align}
Taking the difference of \eqref{eq:derivativex1} and \eqref{eq:derivativex2} gives
\begin{gather}\label{eq:difspeed}
\frac{d}{dt}S_\e=4(p_2^2-p_1^2)\int_{-\infty}^{\infty}\rho_{\e}(y) \big[f_1^{\e}(y)f_2^{\e}(y)-f_1^{\e}(S_\e+y)f_2^{\e}(S_\e+y)\big]dy.
\end{gather}
We have the following useful proposition.
\begin{proposition}\label{pro:speed}
For any $s>0$, we have
\begin{align}\label{eq:limitspeed}
\lim_{\e\to0}4\int_{-\infty}^{\infty}\rho_{\e}(x) \big[f_1^{\e}(x)f_2^{\e}(x)-f_1^{\e}(s+x)f_2^{\e}(s+x)\big]dx= \frac{1}{6}.
\end{align}
The above convergence is uniform about $s\in [\delta,+\infty)$ for any $\delta>0$.
\end{proposition}

\begin{proof}
Let 
$$4\int_{-\infty}^{\infty}\rho_{\e}(x) \big[f_1^{\e}(x)f_2^{\e}(x)-f_1^{\e}(s+x)f_2^{\e}(s+x)\big]dx=:I_1^\e-I_2^\e,$$
where 
$$I_1^\e:=4\int_{-\infty}^{\infty}\rho_{\e}(x) f_1^{\e}(x)f_2^{\e}(x)dx~\textrm{ and }~I_2^\e:=4\int_{-\infty}^{\infty}\rho_{\e}(x) f_1^{\e}(s+x)f_2^{\e}(s+x)dx.$$
For $I_1^\e$, by changing of variables, we have
\begin{align*}
I_1^\epsilon=\int_{-\infty}^\infty\rho(x)\bigg(\int_{-\infty}^x\rho(y)e^{\e(y-x)}dy\bigg)\bigg(\int^{\infty}_x\rho(y)e^{\e(x-y)}dy\bigg)dx.
\end{align*}
Set 
$$F(x):=\int_{-\infty}^x\rho(y)dy.$$
By Lebesgue Dominated convergence Theorem, we have
\begin{align}\label{eq:I1e}
\lim_{\e\to0}I_1^\epsilon&=\int_{-\infty}^\infty\rho(x)\bigg(\int_{-\infty}^x\rho(y)dy\bigg)\bigg(\int^{\infty}_x\rho(y)dy\bigg)dx\nonumber\\
&=\int_{-\infty}^\infty F'(x)F(x)(1-F(x))dx=\frac{1}{6}.
\end{align}
Similarly, for $I_2^\e$ we have
\begin{align*}
I_2^\epsilon=\int_{-\infty}^\infty\rho(x)\bigg(\int_{-\infty}^{x+\frac{s}{\e}}\rho(y)e^{\e(y-x)-s}dy\bigg)\bigg(\int^{\infty}_{x+\frac{s}{\e}}\rho(y)e^{\e(x-y)+s}dy\bigg)dx.
\end{align*}
When $\delta>0$ and $s\in[\delta,+\infty)$, we have $\displaystyle{\frac{\delta}{\e}\leq \frac{s}{\e}}$. Hence,
\begin{align*}
0<I_2^\epsilon&\leq \int_{-\infty}^\infty\rho(x)\bigg(\int_{-\infty}^{\infty}\rho(y)dy\bigg)\bigg(\int^{\infty}_{x+\frac{s}{\e}}\rho(y)dy\bigg)dx\\
&\leq \int_{-\infty}^\infty\rho(x)\bigg(\int^{\infty}_{x+\frac{\delta}{\e}}\rho(y)dy\bigg)dx.
\end{align*}
Therefore, the following convergence holds uniformly for $s\in[\delta,+\infty)$:
\begin{align}\label{eq:I2e}
\lim_{\e\to0}I_2^\e=0.
\end{align}
Combining \eqref{eq:I1e} and \eqref{eq:I2e} gives \eqref{eq:limitspeed}.
\end{proof}
\begin{proof}[Proof of Theorem \ref{thm:twosticky}]
Let $m_0^N(x)=p_1\delta(x-c_1)+p_2\delta(x-c_2)$ for constants $p_i$ and $c_i$ satisfies
\begin{align}\label{eq:initialassumption}
c_1<c_2~\textrm{ and }~p_1^2>p_2^2.
\end{align}
$x^\e_1(t)$ and $x^\e_2(t)$ are obtained by \eqref{approximateODE2}. From Theorem  \ref{lmm:noncollision}, we have
$x^\e_1(t)<x_2^\e(t)$ for any $t\ge0$. By Theorem \ref{thm:fixedN}, for any $T>0$, there are $x_1(t),x_2(t)\in C([0,T])$ such that
$$x_1^\e(t)\to x_1(t)~\textrm{ and }~x_2^\e(t)\to x_2(t)~\textrm{ in }~C([0,T]),~~\e\to 0.$$ 
Hence, we have
$$x_1(t)\leq x_2(t).$$
By Proposition \ref{pro:Npeakon}, we know that solution given by Theorem \ref{thm:fixedN} is the same as the sticky peakon solution when $t<T_*$.

By \eqref{eq:difspeed} and Proposition \ref{pro:speed}, we can see that for any $0<\delta<\min\big\{c_2-c_1,-\frac{1}{6}(p_2^2-p_1^2)\big\}$, there is a $\e_0>0$ such that when $S_\e(t)\ge\delta$ we have
$$\frac{1}{6}(p_2^2-p_1^2)-\delta<\frac{d}{dt}S_\e(t)<\frac{1}{6}(p_2^2-p_1^2)+\delta<0~\textrm{ for any }~\e<\e_0.$$

\textbf{Claim 1:} If there exists $t_0>0$ such that $S_\e(t_0)\leq \delta$, then $S_\e(t)\leq \delta$ for $t>t_0$. Indeed, if there is $t_1>t_0$ and $S_\e(t_1)>\delta$, we set
$$t_2:=\inf\{t<t_1:S_\e(s)>\delta\textrm{ for }s\in(t,t_1)\}.$$
Hence, $t_2\ge t_0$ and $S_\e(t_2)=\delta$. Moreover, $S_\e(t)>\delta$ for $t\in(t_2,t_1)$. Therefore,
\begin{align*}
S_\e(t_1)= \int_{t_2}^{t_1}\frac{d}{ds}S_\e(s)ds +S_\e(t_2)\leq \Big[\frac{1}{6}(p_2^2-p_1^2)+\delta\Big](t_1-t_2)+\delta\leq \delta,
\end{align*}
which is a contradiction with $S_\e(t_1)>\delta$.

\textbf{Claim 2:} We have $S_\e(t)\leq\delta$ for $t\geq \frac{6(c_2-c_1-\delta)}{p_1^2-p_2^2-6\delta}=:t_\delta$. If not, from Claim 1 we have $S_\e(t)>\delta$ for $t\leq t_\delta$. Hence,
\begin{align*}
S_\e(t_\delta)&= \int_0^{t_\delta}\frac{d}{ds}S_\e(s)ds +c_2-c_1\leq \Big[\frac{1}{6}(p_2^2-p_1^2)+\delta\Big]t_\delta+c_2-c_1\leq \delta,
\end{align*}
which is a contradiction.

With the above claims, we can obtain
\begin{align}
\lim_{\e\to0}S_\e(t)=0~\textrm{ for }~t\geq\frac{6(c_2-c_1)}{p_1^2-p_2^2},
\end{align}
which implies \eqref{eq:collision}

\end{proof}
\begin{remark}
Though the peakons are not physical particles and they are not governed by Newton's laws. We have the analogy of the conservation of momentum during the collision. Let $p$ be the `mass' of the peakon. The speeds of the two peakons before collision are $\frac{1}{6}p_1^2+\frac{1}{2}p_1p_2$ and $\frac{1}{6}p_2^2+\frac{1}{2}p_1p_2$ respectively. The speed after collision is $\frac{1}{6}(p_1+p_2)^2$. We can check formally that
\[
(p_1+p_2)\frac{1}{6}(p_1+p_2)^2=p_1\left(\frac{1}{6}p_1^2+\frac{1}{2}p_1p_2\right)+p_2\left(\frac{1}{6}p_2^2+\frac{1}{2}p_1p_2\right).
\]
We can then introduce the instantaneous (infinite) ``force'' as
\[
F_1=p_1[\dot{x}_1] \delta(t-T_*)=\frac{1}{6}p_1p_2(p_2-p_1)\delta(t-T_*)
\]
where $[\dot{x}_1]$ represents the jump of $\dot{x}$ at $t=T_*$. Similarly,
\[
F_2=p_2[\dot{x}_2]\delta(t-T_*)=\frac{1}{6}p_2p_1(p_1-p_2)\delta(t-T_*)
\]
Here $F_1+F_2=0$, which is equivalent to the ``local conservation
of momentum''. 
\end{remark}

\subsection{Discussion about three particle system}

When $N\ge 3$, the limiting $N$-peakon solutions obtained by Theorem \ref{thm:fixedN} can be complicated. 
In this subsection, we study three peakon trajectory interactions.

Denote the  initial data $x_1(0)<x_2(0)<x_3(0)$ and constant amplitudes of peakons $p_i>0$, $i=1,2,3$. Let $x^\e_i(t)$, $i=1,2,3$, be solutions to the regularized system \eqref{approximateODE2} and $x_i(t)$, $i=1,2,3$, be the limiting trajectories given by Theorem \ref{thm:fixedN}. Let $x^s_i(t)$, $i=1,2,3$, be trajectories to sticky peakon solutions given in \cite{GaoLiu}. Before the first collision time, by Proposition \ref{pro:Npeakon} we know that $x_i(t)=x_i^s(t)$, $i=1,2,3$, which is the solution to \eqref{eq:Npeakonmonoticity0}. However, after collisions, the limiting trajectories $x_i(t)$ may or may not coincide with the sticky trajectories $x_i^s(t)$. 
Below, we consider two typical cases.

\textbf{Sticky case (i).} We illustrate this case by an example with $p_1=4,~p_2=2,~p_3=1$ and $x_1(0)=-7,~x_2(0)=-5,~x_3(0)=-3$ (see Figure \ref{fig:threetrajectories1}). For the sticky trajectories (red dashed lines in Figure \ref{fig:threetrajectories1}) $x_i^s(t)$, $i=1,2,3$, the first collision happens between $x_2^s(t)$ and $x_3^s(t)$ at time $t_1^*$. Then $x_2^s(t)$ and $x_3^s(t)$ sticky together traveling with new amplitude $p_2+p_3$ for $t\in (t_1^*,t_2^*)$. Because $p_1>p_2+p_3$, $x_1^s(t)$ catches up with $x_2^s(t)$ and $x_3^s(t)$ at $t_2^*$. At last, the three peakons all sticky together after $t_2^*$.  

When $\e>0$ is small, the behavior of trajectories $x_i^\e(t)$, $i=1,2,3$, given by the regularized system \eqref{approximateODE2} is very similar to the sticky trajectories (see blue solid lines in Figure \ref{fig:threetrajectories1}). This indicates that $x_i(t)\equiv x_i^s(t)$ for any $t>0$ and the limiting peakon solution given by Theorem \ref{thm:fixedN} agrees with the sticky peakon solution.
\begin{figure}[H]
\begin{center}
\includegraphics[width=0.6\textwidth]{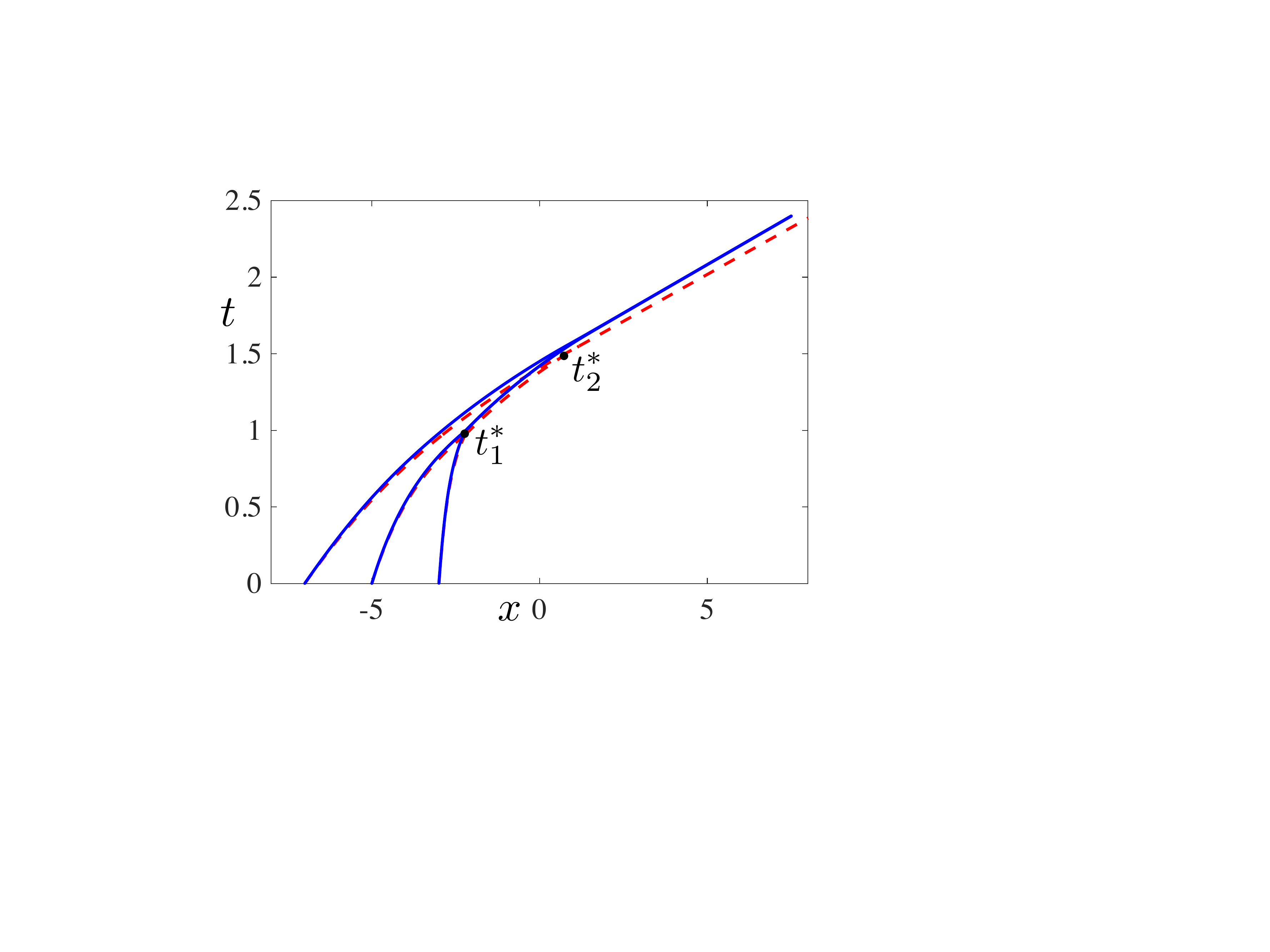}
\end{center}
\caption{$p_1=4,~p_2=2,~p_3=1$ and $x_1(0)=-7,~x_2(0)=-5,~x_3(0)=-3$; $\e=0.02$. 
The blue lines are trajectories of three peakons $\{x^{\e}_i(t)\}_{i=1}^3$ given by dispersive regularization system \eqref{approximateODE2}. The red dashed lines are trajectories of sticky three peakons.}
\label{fig:threetrajectories1}
\end{figure}

\textbf{Sticky and separation case (ii).} We illustrate this case by an example with $p_1=4,~p_2=2,~p_3=3$ and $x_1(0)=-7,~x_2(0)=-6,~x_3(0)=-2$ (see Figure \ref{fig:threetrajectories2}). For the sticky trajectories (red dashed lines in Figure \ref{fig:threetrajectories2}) $x_i^s(t)$, $i=1,2,3$, the first collision happens between $x_1^s(t)$ and $x_2^s(t)$ at time $\hat{t}_1$. Then $x_1^s(t)$ and $x_2^s(t)$ sticky together traveling with new amplitude $p_1+p_2$ for $t\in (\hat{t}_1,\hat{t}_2)$. Because $p_1+p_2>p_3$, $x_1^s(t)$ and $x_2^s(t)$ catch up with $x_3^s(t)$ at $\hat{t}_2$. At last, the three peakons all sticky together after $\hat{t}_2$.  

When $\e>0$ is small, the behavior of trajectories $x_i^\e(t)$, $i=1,2,3$, given by the regularized system \eqref{approximateODE2} is very similar with the sticky trajectories $x_i^s(t)$ before $T_1$, where $x_1^\e(t)$ get close to $x_2^\e(t)$. However, when $x_3^\e(t)$ comes close to $x^\e_2(t)$, $x^\e_2(t)$ separates from  $x_1^\e(t)$ around $T_1$ and gradually moves to $x_3^{\e}(t)$ and then holds together with $x_3^{\e}(t)$. Since $p_2+p_3>p_1$, $x_2^{\e}(t)$ and $x_3^{\e}(t)$ get far away from $x_1^{\e}(t)$.

This indicates the limiting trajectories $x_i(t)\neq x_i^s(t)$ for $t\gtrsim T_1$ and the limiting peakon solution given by Theorem \ref{thm:fixedN} does not agree with the sticky peakon solution. Below, we give some discussions about this interesting phenomenon.

\begin{figure}[H]
\begin{center}
\includegraphics[width=0.6\textwidth]{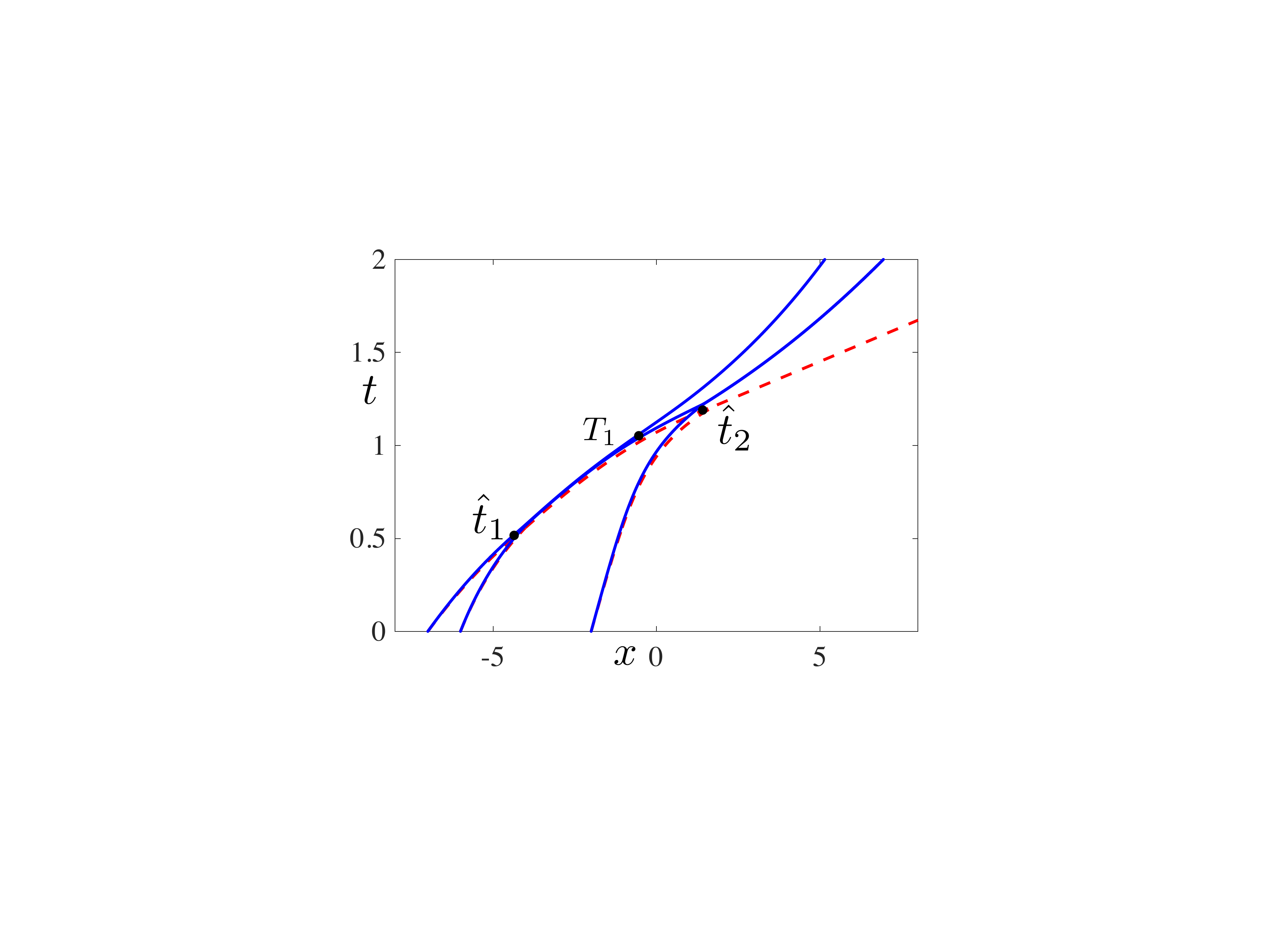}
\end{center}
\caption{$p_1=4,~p_2=2,~p_3=3$ and $x_1(0)=-7,~x_2(0)=-6,~x_3(0)=-2$; $\e=0.02$. The blue lines are trajectories for three peakons $\{x^{\e}_i(t)\}_{i=1}^3$ obtained by dispersive regularization system \eqref{approximateODE2}. The red dashed lines are trajectories of sticky three peakons.}
\label{fig:threetrajectories2}
\end{figure}

Next, we discuss in detail the limiting solution in cases like Figure \ref{fig:threetrajectories2}, i.e. $p_1>p_2>0$, $p_1+p_2>p_3>0$ , $p_1<p_2+p_3$ and $x_3(0)-x_2(0) \gg x_2(0)-x_1(0)>0$. Consider the limiting solution of the form:
\[
u(x, t)=\sum_{i=1}^3 p_iG(x-x_i(t)),
\]
where $x_i(t)$ are Lipschitz continuous and $x_1(t)\le x_2(t)\le x_3(t)$. Since $x_1(0)<x_2(0)<x_3(0)$, by Proposition \ref{pro:Npeakon}, $x_i(t): i=1,2,3$ satisfy the following system for $t\in (0, T_*)$ where $T_*>0$ is the first collision time:
\begin{gather}\label{eq:threepeakon}
\left\{
\begin{split}
&\frac{dx_1}{dt}=\frac{1}{6}p_1^2+\frac{1}{2}p_1p_2e^{-(x_2-x_1)}+\frac{1}{2}p_1p_3e^{-(x_3-x_1)},\\
&\frac{dx_2}{dt}=\frac{1}{6}p_2^2+\frac{1}{2}p_1p_2e^{-(x_2-x_1)}+\frac{1}{2}p_2p_3e^{-(x_3-x_2)}+p_1p_3e^{-(x_3-x_1)},\\
&\frac{dx_3}{dt}=\frac{1}{6}p_3^2+\frac{1}{2}p_1p_3e^{-(x_3-x_1)}+\frac{1}{2}p_2p_3e^{-(x_3-x_2)}.
\end{split}\right. 
\end{gather}
Let $S_i:=x_{i+1}-x_i\ge 0$, $i=1,2$. From \eqref{eq:threepeakon}, the distances $S_i$ satisfy the following equations for $t<T_*$:
\begin{gather}\label{eq:relative}
\left\{
\begin{split}
\frac{dS_1}{dt}=\frac{1}{6}(p_2^2-p_1^2)+\frac{1}{2}p_2p_3e^{-S_2}+\frac{1}{2}p_1p_3 e^{-(S_1+S_2)},\\
\frac{dS_2}{dt}=\frac{1}{6}(p_3^2-p_2^2)-\frac{1}{2}p_1p_2e^{-S_1}-\frac{1}{2}p_1p_3e^{-(S_1+S_2)}.
\end{split}
\right.
\end{gather}
For the case in Figure \ref{fig:threetrajectories2} to happen, $S_2(0)$ should be large enough so that $S_1(T_*)=0$ and 
\[
\lim_{t\to T_*^-}\frac{dS_1}{dt}=\frac{1}{6}(p_2^2-p_1^2)+\frac{1}{2}p_2p_3e^{-S_2(T_*)}+\frac{1}{2}p_1p_3e^{-S_2(T_*)}<0.
\]
In other words, $S_2(T_*)>S_2^*>0$, where $S_2^*$ is defined by:
\[
\frac{1}{6}(p_2^2-p_1^2)+\frac{1}{2}p_2p_3e^{-S_2^*}+\frac{1}{2}p_1p_3e^{-S_2^*}=0.
\]
Since $S_1(t)\ge 0$, while 
\[
\frac{1}{6}(p_2^2-p_1^2)+\frac{1}{2}p_2p_3e^{-S_2}+\frac{1}{2}p_1p_3 e^{-(S_1+S_2)}<0,
\]
\eqref{eq:relative} must not be valid for $t\in (T_*, T_*+\delta)$ for some $\delta>0$ and neither does \eqref{eq:threepeakon}.  Indeed, the new system of equations must be 
\eqref{eq:Npeakonmonoticity0} for $N=2$:
\begin{gather}\label{eq:mergedtwo}
\left\{
\begin{split}
&\frac{d}{dt}x_i(t)=\frac{1}{6}(p_1+p_2)^2+\frac{1}{2}(p_1+p_2)p_3e^{x_i(t)-x_3(t)},~i=1,2,\\
&\frac{d}{dt}x_3(t)=\frac{1}{6}p_3^2+\frac{1}{2}(p_1+p_2)p_3e^{x_2(t)-x_3(t)}.
\end{split}
\right.
\end{gather}
Hence, $S_1(t)=0$ for $t\in (T_*, T_*+\delta)$ while $S_2(t)$ keeps decreasing because $p_1+p_2>p_3$.

Note that the sticky solutions $x_i^s(t)$ satisfy \eqref{eq:mergedtwo} until $x_1^s(t)=x_2^s(t)=x_3^s(t)$. On the contrary, the simulations indicate that $x_1(t)$ and $x_2(t)$ can split when $x_2(t)<x_3(t)$ and then $\{x_i(t)\}_{i=1}^3$ do not satisfy \eqref{eq:mergedtwo} after the splitting. Define the splitting time $T_1$ as 
\[
T_1=\inf\{t\ge T_*: S_1(t)>0\}.
\]

We claim that $T_1\ge T_2:=\inf\{t>0: S_2(t)=S_2^*\}>T_*$.
Suppose for otherwise $T_1<T_2$, then there exists $\delta>0$ such that $S_1(t)>0$  for $t\in (T_1, T_1+\delta)$ with some small $\delta$, $S_1(T_1)=0$ and $S:=\inf_{t\in (T_1, T_1+\delta)}S_2(t)>S_2^*$. For $t\in(T_1, T_1+\delta)$, $S_1$ and $S_2$ must satisfy \eqref{eq:relative} by Proposition \ref{pro:Npeakon}. Consequently, 
\[
\frac{d}{dt}S_1(t)\le \frac{1}{6}(p_2^2-p_1^2)+\frac{1}{2}p_2p_3e^{-S}+\frac{1}{2}p_1p_3e^{-S}<0,~t\in(T_1, T_1+\delta).
\]
Since $S_1(T_1)=0$, we must have $S_1(t)\leq 0$ for $t\in(T_1, T_1+\delta)$. This is a contradiction. 

Now that \eqref{eq:mergedtwo} holds on $(T_*, T_1)$ while $T_1\ge T_2$, we find
\[
T_2=T_*+6(S_2(T_*)-S_2^*)/((p_1+p_2)^2-p_3^2)>T_*.
\]
The question is that when the split happens (i.e. how large can $T_1$ be).
\begin{conjecture*}
At the point of splitting ($t=T_1$), both $x_1(t)$ and $x_2(t)$ are right-differentiable, and $x_1(t): t\ge T_1$ and $x_2(t): t\ge T_1$ are tangent at $t=T_1$. 
\end{conjecture*}
If this conjecture is valid, then we must have
\[
\lim_{t\to T_1^+}\frac{d}{dt}S_1(t)=0
\]
and therefore
\[
T_1=T_2.
\]

In summary, the dispersive regularization limit weak solution is quite different from the sticky particle model in \cite{GaoLiu} when $N\ge 3$. Another difference we note is that the sticky particle model has bifurcation instability for the dynamics of three peakon system:
consider a three particles system with initial data: $p_1=4, x_1(0)=-4$, $p_2=3, x_2(0) \in (-4, 4)$ and $p_3=2, x_3(0)=4$. There exists $x_c\in (-4, 4)$ such that in the $x_2(0)>x_c$ cases, the second and third peakons merge first and then they move apart from the first one (see Figure \ref{fig:unstable} (b)), while $x_2(0)<x_c$ implies that the first two merge first and then they catch up with the third one, merging into a single particle (see Figure \ref{fig:unstable} (a)). This is a kind of bifurcation instability due to the initial position of the second peakon: a little change in $x_2(0)$ results in very different solutions at later time. It seems that the $\e\to 0$ limit does not possess such instability due to the splitting as in Figure \ref{fig:threetrajectories2}. 

\begin{figure}[H]
\begin{center}
\includegraphics[width=0.8\textwidth]{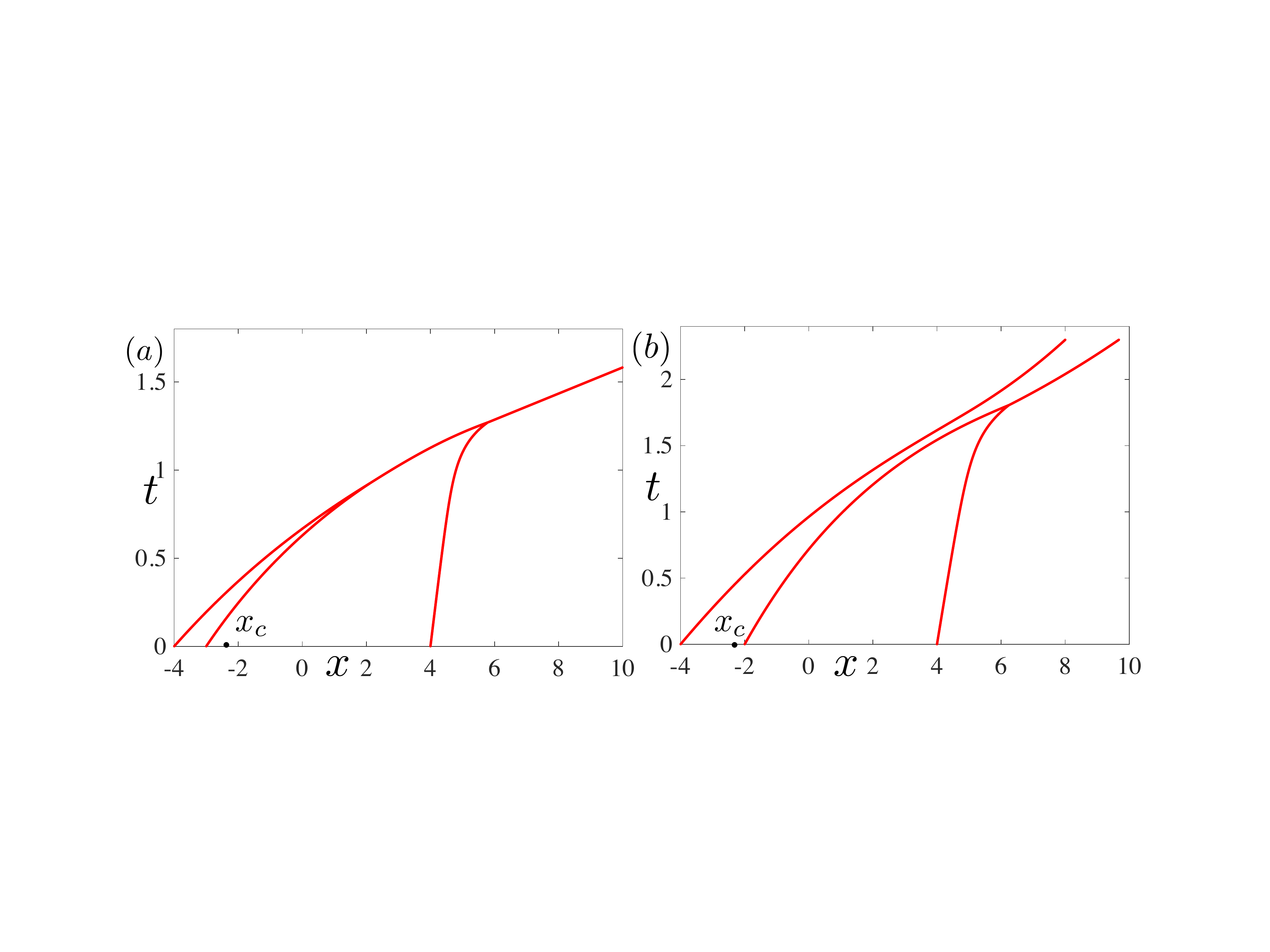}
\end{center}
\caption{(a). $p_1=4,~p_2=3,~p_3=2$ and $x_1(0)=-4,~x_2(0)=-3,~x_3(0)=4$. The three peakons merge into one peakon. (b). $p_1=4,~p_2=3,~p_3=2$ and $x_1(0)=-4,~x_2(0)=-2,~x_3(0)=4$.  The three peakons merge into two separated peakons.}
\label{fig:unstable}
\end{figure}

\section{Mean field limit}\label{sec:meanfield}

In this section, we use a particle blob method to prove global existence of weak solutions to the mCH equation for general initial data $m_0\in\mathcal{M}(\mathbb{R})$. 

Assume that the initial date $m_0$ satisfies
\begin{align}\label{assumem0}
m_0\in\mathcal{M}(\mathbb{R}),\quad\mathrm{supp}\{m_0\}\subset(-L,L),\quad M_0:=\int_{\mathbb{R}}d|m_0|<+\infty.
\end{align}
Let us choose the initial data $\{c_i\}_{i=1}^N$ and $\{p_i\}_{i=1}^N$ to approximate $m_0(x)$. 
Divide the interval $[-L,L]$ into $N$ non-overlapping sub-interval $I_j$ by using the uniform grid with size $h=\frac{2L}{N}$. We choose $c_i$ and $p_i$ as
\begin{align}\label{initialchoose}
c_i:=-L+(i-\frac{1}{2})h;\quad p_i:=\int_{[c_i-\frac{h}{2},c_i+\frac{h}{2})}dm_0,\quad i=1,2,\cdots,N. 
\end{align}
Hence, we have
\begin{align}\label{eq:sumproperty}
\sum_{i=1}^N|p_i|\leq \int_{[-L,L]}d|m_0|\leq M_0.
\end{align}
Using \eqref{initialchoose}, one can easily prove that $m_0$ is approximated by
\begin{align}\label{mNx0}
m^N_0(x):=\sum_{j=1}^N p_j\delta(x-c_j)
\end{align}
 in the sense of measures.  Actually, for any test function $\phi\in C_b(\mathbb{R})$,  we know $\phi$ is uniformly continuous on $[-L,L]$. Hence, for any $\eta>0$, there exists a $\delta>0$ such that when $x,y\in[-L,L]$ and $|x-y|<\delta$, we have $|\phi(x)-\phi(y)|<\eta$. Hence, choose $\frac{h}{2}<\delta$ and we have
\begin{align}\label{initialerror}
&\bigg|\int_{\mathbb{R}}\phi(x)dm_0-\int_{\mathbb{R}}\phi(x)dm^N_0\bigg|=\bigg|\int_{[-L,L]}\phi(x)dm_0-\int_{[-L,L]}\phi(x)dm^N_0\bigg|\nonumber\\
=&\bigg|\sum_{i=1}^N\int_{[c_i-\frac{h}{2},c_i+\frac{h}{2})}\big(\phi(x)-\phi(c_i)\big)dm_0\bigg|\leq \eta\sum_{i=1}^N\int_{[c_i-\frac{h}{2},c_i+\frac{h}{2})}d|m_0|\leq M_0\eta.
\end{align}
Let $\eta\rightarrow0$ and we obtain the narrow convergence from $m^N_0(x)$ to $m_0(x)$.

For initial data $m_0^N(x)$, Theorem \ref{thm:fixedN} gives a weak solution $u^N(x,t)=\sum_{i=1}^Np_iG(x-x_i(t)),$
where $x_i(0)=c_i$ and $p_i$ are given by \eqref{initialchoose}. Moreover, \eqref{eq:lipschitzxj} holds for $x_i(t)$, $1\leq i\leq N.$

Next, we are going to use some space-time BV estimates to show compactness of $u^{N}$. To this end, we recall the definition of BV functions.
\begin{definition}\label{BV}
(i).
For dimension $d\geq1$ and an open set $\Omega\subset\mathbb{R}^d$, a function $f\in L^1(\Omega)$ belongs to $BV(\Omega)$ if
$$Tot.Var.\{f\}:=\sup\Big\{\int_{\Omega}f(x)\nabla\cdot\phi(x)dx:\phi\in C_c^1(\Omega;\mathbb{R}^d), ~~||\phi||_{L^\infty}\leq1\Big\}<\infty.$$

(ii). (Equivalent definition for one dimension case) A function $f$ belongs to $BV(\mathbb{R})$ if for any $\{x_i\}\subset\mathbb{R}$, $x_i<x_{i+1}$, the following statement holds:
$$Tot.Var.\{f\}:=\sup_{\{x_i\}}\Big\{\sum_i|f(x_i)-f(x_{i-1})|\Big\}<\infty.$$
\end{definition}
\begin{remark}
Let $\Omega\subset\mathbb{R}^d$ for $d\geq1$ and $f\in BV(\Omega)$. $Df:=(D_{x_1}f,\ldots,D_{x_d}f)$ is the distributional gradient of $f$. Then, $Df$ is a vector Radon measure and  the total variation of $f$ is equal to the total variation of $|D f|$: $Tot.Var.\{f\}=|D f|(\Omega).$  Here, $|D f|$ is the total variation measure of the vector measure $Df$ (\cite[Definition (13.2)]{Leoni}).

If a function $f:\mathbb{R}\rightarrow \mathbb{R}$ satisfies Definition \ref{BV} (ii), then $f$ satisfies Definition (i). On the contrary, if $f$ satisfies Definition \ref{BV} (i), then there exists a right continuous representative which satisfies Definition (ii). See \cite[Theorem 7.2]{Leoni} for the proof.
\end{remark}

Now, we give some space and time BV estimates about $u^{N},\partial_xu^{N}$, which is similar to \cite[Proposition 3.3]{GaoLiu}.
\begin{proposition}\label{BV properties}
Assume initial value $m_0$ satisfies \eqref{assumem0}.  $p_i$ and $c_i$, $1\leq i\leq N$,  are given by \eqref{initialchoose} and $m_0^N$ is defined by \eqref{mNx0}. Let $u^N(x,t)=\sum_{i=1}^Np_iG(x-x_i(t))$ be the $N$-peakon solution given by Theorem \ref{thm:fixedN} subject to initial data $m^N(x,0)=(1-\partial_{xx})u^N(x,0)=m_0^N(x)$. Then, the following statements hold.

$\mathrm{(i).}$
For any $t\in[0,\infty)$, we have
\begin{align}
Tot.Var.\{u^N(\cdot,t)\}\leq M_0,\quad Tot.Var.\{\partial_xu^N(\cdot,t)\}\leq 2M_0\textrm{ uniformly in }N.\label{total variation}
\end{align}

$\mathrm{(ii).}$
\begin{align}
||u^N||_{L^\infty}\leq \frac{1}{2}M_0,\quad ||\partial_xu^N||_{L^\infty}\leq \frac{1}{2}M_0\textrm{ uniformly in }N.\label{bounded}
\end{align}

$\mathrm{(iii).}$
For $t,s\in[0,\infty)$, we have
\begin{small}
\begin{align}
\int_{\mathbb{R}}|u^N(x,t)-u^N(x,s)|dx\leq \frac{1}{2}M_0^3|t-s|,~~
\int_{\mathbb{R}}|\partial_xu^N(x,t)-\partial_xu^N(x,s)|dx\leq M_0^3|t-s|.\label{lipschitz}
\end{align}
\end{small}

$\mathrm{(iv).}$
For any $T>0$, there exist  subsequences of $u^{N}$, $u_x^{N}$ (also labeled as $u^{N}$, $u_x^{N}$) and two functions $u,~u_x\in BV(\mathbb{R}\times[0,T))$ such that
\begin{align}\label{BVconvergence1}
u^{N}\rightarrow u,~~u_x^{N}\rightarrow u_x\textrm{ in }L_{loc}^1(\mathbb{R}\times[0,+\infty)) \textrm{ as } N\rightarrow\infty,
\end{align}
and  $u$, $u_x$ satisfy all the properties in  $\mathrm{(i)}$, $\mathrm{(ii)}$ and $\mathrm{(iii)}$.

\end{proposition}
\begin{proof}
See \cite[Proposition 3.3]{GaoLiu}. We remark that the key estimate to prove \eqref{lipschitz} is \eqref{eq:lipschitzxj}.
\end{proof}

With Proposition \ref{BV properties}, we have the following theorem:
\begin{theorem}\label{globalweaksolution}
Let the assumptions in Proposition \ref{BV properties} hold. Then, the following statements hold:

$\mathrm{(i)}$. The limiting function $u$ obtained in Proposition \ref{BV properties} $(\mathrm{(iv)})$  satisfies
\begin{align}\label{solutionspace}
u\in C([0,+\infty);H^1(\mathbb{R}))\cap L^\infty(0,+\infty; W^{1,\infty}(\mathbb{R}))
\end{align}
and it is a global weak solution of the mCH equation \eqref{mCH}.

$\mathrm{(ii)}$.  For any $T>0$, we have
$$m=(1-\partial_{xx})u\in \mathcal{M}(\mathbb{R}\times[0,T))$$
and there exists a subsequence of $m^{N}$ (also labeled as $m^{N}$) such that
\begin{align}\label{mconvergence}
m^{N}\stackrel{\ast}{\rightharpoonup} m\textrm{ in } \mathcal{M}(\mathbb{R}\times[0,T))\quad (\textrm{as }N\rightarrow+\infty ).
\end{align}

$\mathrm{(iii)}$.  For a.e. $t\geq0$ we have (in subsequence sense)
\begin{align}\label{aetimeconvergence}
m^N(\cdot,t)\stackrel{\ast}{\rightharpoonup} m(\cdot,t)\textrm{ in } \mathcal{M}(\mathbb{R})~\textrm{as }~ N\rightarrow+\infty 
\end{align}
and
\begin{align}\label{compactsupport}
\mathrm{supp}\{m(\cdot,t)\}\subset\Big(-L-\frac{1}{2}M_0^2t, L+\frac{1}{2}M_0^2t\Big),
\end{align}

\end{theorem}

\begin{proof}
The proof is similar to \cite[Theorem 3.4]{GaoLiu} and we omit it.
\end{proof}
\begin{remark}\label{positiveRadonmeasure}
We remark that when $m_0$ is a positive Radon measure, $m$ is also positive. Actually, $m_0\in\mathcal{M}_+(\mathbb{R})$ implies that $p_i\geq0$ and $m^{N,\epsilon}\geq0.$ Therefore, the limiting measure $m$ belongs to $\mathcal{M}_+(\mathbb{R}\times[0,T)).$ 

By the same methods as in \cite[Theorem 3.5]{GaoLiu}, we can also show that for a.e. $t\geq0$,
\begin{align}\label{eq:totalBVstability}
m(\cdot,t)(\mathbb{R})=m_0(\mathbb{R}),~~|m(\cdot,t)|(\mathbb{R})\leq |m_0|(\mathbb{R}).
\end{align}

\end{remark}

\newpage

\bibliographystyle{plain}
\bibliography{bibofblobmCH}

\end{document}